\documentclass[onecolumn]{IEEEtran}
\usepackage{amsmath,amsfonts,amssymb,euscript, graphicx, 
epsfig,
enumerate,float,afterpage, subfig, ifthen}%

\usepackage{url}
\usepackage{hyperref}
\usepackage{algpseudocode}

\newtheorem{thm}{Theorem}

\newtheorem{lem}{Lemma}

\newtheorem{assumption}{Assumption}

\newcommand{\expect}[1]{\mathbb{E}\left[#1\right]}

\newcommand{\norm}[1]{||{#1}||}

\newcommand{\script}[1]{{{\cal{#1} }}}

\begin{document}

\title{Opportunistic Learning for Markov Decision Systems with Application to Smart Robots}

\author{Michael J. Neely \\ University of Southern California\\ \url{https://viterbi-web.usc.edu/~mjneely/}
}


\maketitle


\begin{abstract} 
This paper presents an online method that learns optimal decisions 
for a discrete time Markov decision problem with an opportunistic 
structure. The state at time $t$ is a pair $(S(t),W(t))$ where $S(t)$
takes values in a finite set $\script{S}$ of \emph{basic states}, and 
$\{W(t)\}_{t=0}^{\infty}$ is an i.i.d. sequence of random vectors that affect
the system and that have an unknown distribution. Every slot $t$ the controller observes $(S(t),W(t))$
and chooses a control action $A(t)$. The triplet $(S(t),W(t),A(t))$ determines 
a vector of costs and  the transition probabilities for the next
state $S(t+1)$.   The goal is to 
minimize the time average of an objective function subject to additional time
average cost constraints. We develop an algorithm that acts on a corresponding 
virtual system where $S(t)$ is replaced by a decision variable.  An equivalence between virtual and actual systems is established by enforcing a collection of time averaged global balance equations. For any desired $\epsilon>0$, we prove the algorithm achieves 
an $\epsilon$-optimal solution on the virtual
system with a convergence time of $O(1/\epsilon^2)$.  The actual system runs at the same time, its actions are informed by the virtual system,  and its conditional transition probabilities and costs are proven to be the same as the virtual system at every instant of time. Also, its unconditional probabilities and costs are shown in simulation to closely match the virtual system. Our simulations consider online control of a robot that explores a region of interest. Objects with varying rewards appear and disappear and the robot learns what areas to explore and what objects to collect and deliver to a home base.  
\end{abstract} 

\section{Introduction} 

This paper considers a Markov decision system that 
operates in slotted time $t \in \{0, 1, 2, \ldots\}$. The state of the system is given by a pair $(S(t), W(t))$, where $S(t)$ takes values in a finite set $\script{S}=\{1, \ldots, n\}$ of \emph{basic states} (where $n$ is a positive integer), 
and $\{W(t)\}_{t=0}^{\infty}$ is a sequence of independent and identically distributed (i.i.d.) random vectors of arbitrarily large dimension that take values in a (possibly infinite)  set $\script{W}$.  The value $W(t)$ can represent a random fluctuation or augmentation of the state of the system, such as a random vector of rewards,  costs, or   side information.   The distribution of $W(t)$ is unknown to the system controller. This is an \emph{opportunistic Markov decision problem} because the controller can observe the value of $W(t)$ at the start of slot $t$ and can use this knowledge to inform its action.  Specifically, every slot $t$ the controller observes $(S(t),W(t))$ and chooses an action $A(t)$ from an action set $\script{A}$.   The triplet $(S(t),W(t),A(t))$ determines 
a vector of costs incurred on slot $t$ and also the transition probability associated with the next basic state $S(t+1)$.   

It shall be useful to assume the action has the form $A(t) = (A_1(t), \ldots, A_n(t))$, where 
$A_i(t)$ is a \emph{contingency action} given that $S(t)=i$. Assume $\script{A}=\script{A}_1\times\cdots\times \script{A}_n$, where $\script{A}_i$ is the action set when $S(t)=i$.  
For each pair of basic states $i,j\in \script{S}$ define a \emph{transition probability function}  $p_{i,j}(w,a_i)$ so that 
\begin{equation} \label{eq:pij}
P[S(t+1)=j|S(t)=i, W(t),A(t),H(t)]=p_{i,j}(W(t),A_i(t))
\end{equation} 
where $H(t)$ is the system history before slot $t$. The system has the Markov property because the value $S(t+1)$ is conditionally independent of history $H(t)$ given the current $(S(t),W(t),A(t))$.   

Fix $k$ as a nonnegative integer. For $i \in \script{S}$ and $l \in \{0, 1, \ldots, k\}$ define \emph{cost functions} $c_{i,l}(w,a_i)$. Define the cost vector for slot $t$ by $C(t)=(C_0(t), C_1(t), \ldots, C_k(t))$ where 
\begin{equation} \label{eq:cost-function}
C_l(t) = \sum_{i\in \script{S}} 1_{\{S(t)=i\}}c_{i,l}(W(t),A_i(t))
\end{equation} 
for $l \in \{0, \ldots, k\}$, 
where $1_X$ is an indicator function that is $1$ when event $X$ is true and $0$ else. The goal is 
to make decisions over time to produce random processes $\{S(t)\}_{t=0}^{\infty}$ and $\{C(t)\}_{t=0}^{\infty}$ that solve the following time average optimization problem: 
\begin{align}
\mbox{Minimize:} & \quad \overline{C}_0 \label{eq:p1}\\
\mbox{Subject to:} & \quad \overline{C}_l \leq 0 \quad \forall l \in \{1, \ldots, k\} \label{eq:p2} \\
&\quad A(t) \in \script{A}   \quad \forall t \in \{0, 1, 2, \ldots\} \label{eq:p3}
\end{align}
where $\overline{C}_l$ denotes the limiting time average 
\begin{align*}
\overline{C}_l=\limsup_{T\rightarrow\infty} \frac{1}{T}\sum_{t=0}^{T-1}C_l(t)\quad \forall l \in \{0, 1, \ldots, k\} 
\end{align*}
The problem is assumed to be \emph{feasible}, meaning there is a sequence of  actions $A(t)\in \script{A}$ for $t \in \{0, 1, 2, \ldots\}$ that satisfy a \emph{causal and measurable} property (specified in Section \ref{section:control}) and such that constraint \eqref{eq:p2} holds in an almost sure sense.  

There are two main challenges: First, the dimension of $W(t)$ can be large and its corresponding set $\script{W}$ can be infinite, so the full state space $\script{S} \times \script{W}$ is overwhelming. 
Second, the distribution of the i.i.d. random vectors $\{W(t)\}_{t=0}^{\infty}$ is unknown. 
It  is not always possible 
to estimate the distribution in a timely manner. This paper develops a low complexity algorithm that learns to make efficient decisions that drive the system close to optimality. Our algorithm depends  on $n$, the number of basic states, and its implementation and convergence time is independent of the dimension of $W(t)$ and the size of $\script{W}$. The idea is that, rather than learn the full distribution, it is sufficient to learn
certain  \emph{max weight functionals}.  The algorithm can be viewed as a Markov-chain based generalization of the drift-plus-penalty algorithm in \cite{sno-text} for opportunistic network scheduling.

This paper focuses simulations on a toy example of a roving robot, described in the next subsection. Other applications that have this opportunistic Markov decision structure include: 

\begin{itemize} 
\item Wireless scheduling: Consider a mobile wireless network where channel qualities over multiple antennas can be measured before each use.  Then $W(t)$ is a vector of measured attenuations or fading states on each channel at time $t$. Knowledge of $W(t)$
informs which channels should be used. The basic state $S(t)$ can represent location or activity states that change according to a Markov decision structure. 
\item Transportation scheduling: Consider a driver who repeatedly chooses one of multiple customers to transport about a city. Then $S(t)$ is the current location of the driver, while $W(t)$ is a vector that contains the destination, duration, and cost associated with the current customer options. 
\item Computational processing: Consider a computer that repeatedly processes tasks using one of multiple processing modes. Then $W(t)$ is a vector of parameters specific to the task at time $t$, such as time, energy, and cost information that can be observed and that informs the choice of processing mode. 
\end{itemize}

\subsection{Robot example}  \label{section:robot-example} 

\begin{figure}[htbp]
   \centering
   \includegraphics[width=6.8in]{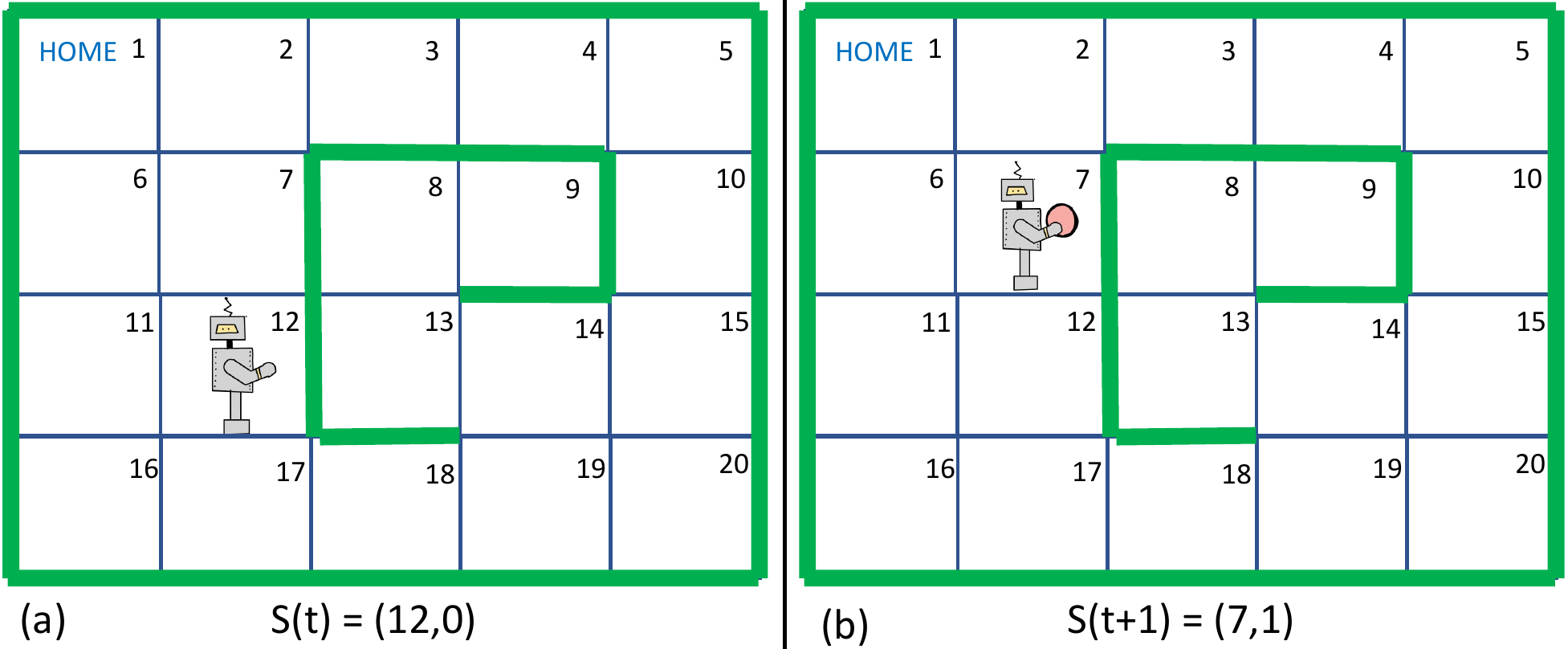} 
   \caption{(a) The robot has current state $S(t)=(12,0)$ because it is in location 12 and is not 
   holding an object. Suppose there is an object at location 12, and the robot decides to collect this object (earning reward $W_{12}(t)$) and then move to location 7; (b) The robot transitions to state $S(t+1)=(7,1)$ because it is now in location 7 and is holding an object. It will continue to hold the object (and hence is blocked from accumulating more rewards by collecting more objects) until it returns to the home location 1 (where it ``deposits'' the object at home and transitions to state $(1,0)$).} 
   \label{fig:robot-fig1}
\end{figure}

This paper focuses on a toy example of a robot that seeks out valuable objects over a region and delivers these  to a home base (see Fig. \ref{fig:robot-fig1}).  The basic state has the structure $S(t) = (Location(t), Hold(t))$ where $Location(t)$ is one of the 20 cells shown in the region, and $Hold(t) \in \{0,1\}$  indicates whether or not the robot is currently holding an object. Thus, there are $20\times 2=40$ basic states ($n=40$). Given $S(t)=(a,h)$ for $a \in \{1, \ldots,20\}$ and $h\in\{0,1\}$, the action $A_{a,h}(t)$ specifies if the robot collects an object and also if it stays in its same cell or moves to an adjacent cell to the North, South, East, or West.   So $A_{a,h}(t)$ is always one of the following 10 actions 
$$\{(Collect, Stay), (Collect, N), (Collect, S), (Collect, W), (Collect, E)\}$$
$$\{(NoCollect,Stay), (NoCollect,N), (NoCollect,S),(NoCollect,W), (NoCollect,E)\}$$ 
Actions within this set of 10 are removed from consideration if they are impossible in state $(a,h)$, such as moving in a direction blocked by a wall, or collecting a new object when the robot is already holding one  (that is, when $h=1$). 

Objects randomly appear and disappear in each location on every slot, and each object has a different value. Define $W(t) = (W_1(t), \ldots, W_{20}(t))$, where $W_a(t)$ is the value of the object in location $a$ at time $t$ (if there is currently no object in location $a$ then $W_a(t)=0$). Assume $\{W(t)\}_{t=0}^{\infty}$ are i.i.d. random vectors with a joint distribution that is unknown to the robot.  On each slot $t$, the robot can view the full reward vector $W(t)$. Then:
\begin{itemize} 
\item If the robot is in location $a \in \{1, \ldots, 20\}$ and is \emph{not holding} an object (so $S(t)=(a,0)$), it chooses whether or not to collect the object (if there is one) in its location $a$. Collecting the object earns reward $W_a(t)$.  The robot also chooses which location to move to next. For example, in Fig. \ref{fig:robot-fig1}(a) the state is $S(t)=(12,0)$ and the robot can choose the next 
location $b\in \{12, 7, 11, 17\}$ (it cannot choose $b=13$ because of the wall).
\item If the robot is in location $a\in\{1, \ldots, 20\}$ and is \emph{holding} an object (so $S(t)=(a,1)$), then it is not allowed to collect another object. Its only choice is which location to visit next. The robot can only change $Hold(t)=1$ to $Hold(t)=0$ by visiting the home base and depositing its object there, where its state changes to $(1,0)$. 
\end{itemize} 

Since every object must be brought home before new ones are collected, the robot must be careful not to waste  time carrying low-valued objects. The goal is to maximize time average reward, equivalent to minimizing $\overline{C}_0$ where: $C_0(t)=-W_{Location(t)}(t)1_{collect}(t)$; $1_{collect}(t)$ is a binary variable that is 1 if and only if the robot collects an object at time $t$; $W_{Location(t)}$ is the reward of that object.  For this robot example, the transition probability functions $p_{(a,h), (b,h')}(\cdot)$  are binary valued and depend only on  $A_{a,h}(t)$, so the next state $(b,h')\in\script{S}$ is  determined by the current state and action. This example has $k=0$, so there are no additional cost constraints $\overline{C}_l\leq 0$.  Section \ref{section:simulation-power} imposes an additional time average power constraint for an extended problem  where different actions of the robot can expend different amounts of power.

Let $r^*$ denote the optimal time average reward that can be achieved, considering all decision strategies. The value $r^*$ (and the resulting optimal strategy) depends on the distribution of the vector $W(t)$ rather than just its mean value $\expect{W(t)} = (\expect{W_1(0)}, \ldots, \expect{W_{20}(0)})$. It is not possible to approximately learn the full distribution in a timely manner. Something more efficient must be done. Further, it is not obvious how to achieve $r^*$ even if the distribution were fully known.

\subsection{Comparing against distribution-aware strategies} \label{section:robot-example-dist}

To illuminate the structure of the robot problem, consider the following distribution on reward vectors $W(t)$: Entries of $W(t)$ are mutually independent; $W_1(t)=0$ surely (no object appears at the home location); for $a \in \{2, \ldots, 20\}$, $W_a(t)=B_a(t)R_a(t)$ with  $B_a(t) \sim Bern(1/2)$ being 1 if and only if an object appears in cell $a$ on slot $t$, $R_a(t)$ is the reward of the object that appears;  $R_9(t)\sim \mbox{Unif}[0,20]$,  $R_{16}(t) \sim \mbox{Unif}[0,4]$, $R_a(t) \sim \mbox{Unif}[0,1]$ for $a \in \{2, \ldots, 20\} \setminus\{9,16\}$.  With these parameters, the most valuable objects tend to appear in location 9, which is the most difficult location to reach (see Fig. \ref{fig:robot-fig1}). The second most valuable location is 16, while all other locations tend to have low valued objects.  

In the algorithm of our paper, the robot must \emph{learn} to return home as quickly as possible after it collects an object.  However, it is useful to compare performance of our algorithm against the following two heuristic strategies that have a frame-based renewal structure and that are fine tuned with knowledge of the problem structure and probability distributions. 

\begin{itemize} 
\item Heuristic 1: Fix a parameter $\theta \in [0, 4)$; Starting from cell 1, move to cell 16 in 3 steps via the path $1\rightarrow 6 \rightarrow 11\rightarrow 16$ (ignoring all rewards  along the way); wait in cell 16 until an object appears with value $W_{16}(t)> \theta$; collect the object and return home in 3 steps; Repeat. Since $R_{16}(t)\sim \mbox{Unif}[0,4]$, we have $P[W_{16}(t)>\theta]=P[B_{16}(t)=1]P[R_{16}(t)> \theta]=\frac{4-\theta}{8}$.  The expected reward over one frame is $\expect{W_{16}(t)|W_{16}(t)> \theta}=\frac{1}{2}(\theta + 4)$. 
By renewal-reward theory, the time average reward of this policy is 
$$ \overline{r} = \frac{\frac{1}{2}(\theta + 4)}{5 +\frac{8}{4-\theta}} = 0.33616$$
where the numerical value is obtained by maximizing over $\theta \in [0,4)$, achieved at 
$\theta^*=1.6808$. 

\item Heuristic 2: Fix a parameter $\theta \in [0, 20)$; Starting from location 1, move to location 9 over any path that takes exactly 10 steps (ignoring all rewards along the way); stay in location 9 until we see an object with value $W_9(t)> \theta$; collect this and return home using any 10-step path. By renewal-reward theory, the time average reward is 
$$ \overline{r} = \frac{\frac{1}{2}(\theta + 20)}{19 + \frac{40}{20-\theta}} = 0.66791$$
where the numerical value is obtained by maximizing over $\theta \in [0, 20)$, 
achieved at  $\theta^*=12.690$. 
\end{itemize} 

With this reward distribution, it is desirable to visit the hard-to-get location 9 to obtain  higher-valued rewards. 
It  is not clear whether or not Heuristic 2 is optimal. However, our simulated results for the algorithm of this paper, implemented online over $10^6$ time slots, yields a time average reward  in a \emph{virtual system} of $\overline{r}=0.6672$,  and in the \emph{actual system} of $\overline{r}=0.6604$ (the concept of \emph{virtual} and \emph{actual} systems is made apparent when the algorithm is presented). This suggests that Heuristic 2 is either optimal or near optimal. It also suggests that our online algorithm learns to visit home so it can refresh its $Hold(t)$ state, learns to ignore low-valued objects, learns the shortest paths to and from location 9, learns near-optimal thresholding rules, all without knowing the distribution of the rewards. The reported values $0.6672$ and $0.6604$ for our algorithm are time averages over $t \in \{0, \ldots, 10^6\}$,  so these averages include  the  relatively small rewards earned early on when the robot  is just starting to learn  efficient behavior. 

\subsection{Prior work}

Our paper characterizes optimality of the stochastic problem in terms of a 
deterministic and nonconvex problem \eqref{eq:det1}-\eqref{eq:det5}. This deterministic problem 
  is reminiscent of 
linear programming representations of optimality for simpler Markov decision
problems that do not have the opportunistic scheduling aspect and that 
have finite state and action sets (see, for example, \cite{ross-prob}\cite{puterman}\cite{mine-mdp}).  
In principle, our
nonconvex problem  \eqref{eq:det1}-\eqref{eq:det5}
can be transformed into a convex problem by a nonlinear change of variables
similar to methods for linear fractional programming in \cite{boyd-convex}\cite{fox-linear-fractional-mdp}. 
The work \cite{fox-linear-fractional-mdp} 
uses the nonlinear transformation for offline computation of an optimal 
Markov decision policy. However, the approaches in \cite{boyd-convex}\cite{fox-linear-fractional-mdp}
do 
not help for our context because: (i) The deterministic problem \eqref{eq:det1}-\eqref{eq:det5} uses
abstract and unknown 
sets $\overline{\Gamma}_i$ that make
the resulting convex problem very complex; (ii)  We seek an online
solution with desirable time averages, and time averages are not preserved under nonlinear
transformations.  Classical descriptions
of optimality for dynamic programming and Markov
decision problems with general Borel spaces are 
in \cite{blackwell-discount-dp}\cite{maitra-dp}\cite{schal-dp}. Our characterization
is different from these classical approaches because it leverages 
the special opportunistic structure. In particular, it 
isolates the basic state variables into a finite set $\script{S}$ and treats the opportunistic
aspect of the problem via the compact and convex sets $\overline{\Gamma}_i$. This is 
useful because it connects directly with our proposed algorithm and enables complexity
and convergence to be determined by the size of the finite set $\script{S}$, independent
of the (possibly infinite) number of additional states added by $W(t)$. 

Our paper uses a classical Lyapunov drift technique pioneered by Tassiulas and Ephremides for 
stabilizing queueing networks \cite{tass-radio-nets}\cite{tass-server-allocation}. Specifically, 
we use an extended  \emph{drift-plus-penalty} method that incorporates
a penalty function to jointly minimize time average cost subject to stability of certain virtual
queues \cite{sno-text}. Such methods have been extensively used for opportunistic scheduling in data
networks with unknown arrival and channel 
probabilities \cite{tass-delayed-info}\cite{neely-energy-it}\cite{neely-fairness-ton}. Other opportunistic scheduling approaches are stochastic Frank-Wolfe methods 
\cite{prop-fair-down}\cite{vijay-allerton02}\cite{neely-frank-wolfe-ton},  fluid
model techniques \cite{stolyar-greedy}, and related dual and primal-dual 
approaches \cite{shroff-opportunistic}\cite{atilla-primal-dual-jsac}\cite{atilla-fairness-ton}\cite{stolyar-gpd-gen}. 

The drift-plus-penalty method was used to treat Markov decision 
problems (MDPs) with an opportunistic scheduling aspect 
in \cite{asynchronous-markov}\cite{neely-fractional-markov-allerton2011}. 
The work \cite{asynchronous-markov} is the most similar to the current paper. That work 
uses drift-plus-penalty theory on  
a virtual system.  However, the algorithm that runs on the virtual system requires knowledge of certain \emph{max-weight functionals} that depend on  unknown probability distributions. For this, it estimates the  
functionals by sampling over a window of past $W(t)$ samples (see also \cite{neely-mwl-tac}).  
This slows down learning time and makes a precise convergence analysis difficult. 
In contrast, the current paper develops a new layered stochastic optimization technique that operates online, on a single timescale, and 
does not require averaging over a window of past samples.

A related problem of robot navigation over a directed graph with randomly generated 
rewards at each node was considered in \cite{robot-routing2017}. There, the robot can accumulate the 
reward for any node it visits (without the constraint of only holding one object at a time). 
They provide NP-hardness
results for finite horizons and approximation results for infinite horizons by relating to a problem of finding a minimum cycle mean  in a weighted graph. Related problems of robot patrolling are considered in 
\cite{robot-patrolling} using a heuristic algorithm based on Jensen-Shannon divergence, and 
in \cite{robot-patrolling2} using  path decomposition and dynamic programming. The works 
\cite{robot-routing2017}\cite{robot-patrolling}\cite{robot-patrolling2} 
assume either static rewards or a  known probability distribution, and do 
not have the same opportunistic learning aspect as the current paper.  An MDP 
approach to a multi-robot environment sensing problem is expored in 
\cite{MDP-robot-opportunistic}. There, a random vector is to be estimated, different robots can observe noisy components of this vector by scanning different regions of the environment, and robots can opportunistically share information when they meet. The problem  is NP-complete and so the paper investigates greedy approximations. 

Online MDPs are 
treated in \cite{online-MDP-2009}\cite{xiaohan-online-MDP-journal}\cite{online-MDP2022}. 
The work \cite{online-MDP-2009} treats a known transition probability model but adversarial
costs that are revealed \emph{after} a decision is made. An $O(\sqrt{T})$ regret algorithm is developed
using online convex programming and a quasi-stationary assumption on the Markov chain. 
The model is extended in \cite{xiaohan-online-MDP-journal} to allow time varying constraint costs and coupled multi-chains, again with $O(\sqrt{T})$ regret,  see also a recent treatment in \cite{online-MDP2022}. 
MDPs where transition probabilities are allowed to vary slowly over time are considered in \cite{epsilon-mdp}. The above works have a different structure from the current paper and do not have an opportunistic learning aspect. Also, the works \cite{online-MDP-2009}\cite{xiaohan-online-MDP-journal}\cite{online-MDP2022} take a 2-timescale approach and transform the decision variables to a vector in a ``policy class,'' which  requires each decision to solve a linear program associated with a stationary distribution for a certain ``time-$t$'' MDP. In contrast, the current paper operates on one timescale and uses an easier ``max-weight'' type decision on every slot.

\subsection{Our contributions} 

Similar to \cite{asynchronous-markov}, our paper focuses on a virtual system. However, our virtual 
system has a simpler structure that directly connects to the actual system.  We prove  the virtual
and actual systems have the same optimal cost, which is described by a deterministic nonconvex optimization problem.  Next, we develop an algorithm where the virtual system observes the current $W(t)$ and chooses a \emph{contingency action} $A_i(t)$ for each $i \in \script{S}$. The virtual system also replaces the actual Markov state variable $S(t)$ with a related decision variable $\pi(t)$ that can be chosen as desired. 
Equivalence between virtual and actual systems is enforced by imposing time averaged global balance constraints. Our algorithm uses a novel layered structure and, unlike \cite{asynchronous-markov}, does not require estimation over a window of past samples. This enables explicit performance guarantees for time average expected cost. Specifically, for any desired $\epsilon>0$, we show the virtual system achieves within $O(\epsilon)$ of optimality with convergence time $1/\epsilon^2$. The complexity and convergence time are independent of the size of the random vector $W(t)$, which is allowed to have an arbitrarily large dimension.  

This virtual algorithm also provides a simple online algorithm for the actual system. We show that for all time $t$, the actual and virtual systems have the same \emph{conditional} transition probability and \emph{conditional} expected cost given the current state. Specifically, with $S^{a}(t)$ and $S^{v}(t)$ being the state of the \emph{actual} and \emph{virtual} systems, respectively, we have 
\begin{align*}
&P[S^{a}(t+1)=j|S^{a}(t)=i] = P[S^{v}(t+1)=j|S^{v}(t)=i]  \\
&\expect{C_l^{a}(t)|S^{a}(t)=i} = \expect{C_l^{v}(t)|S^{v}(t)=i} 
\end{align*}
for all $i,j \in \script{S}$,  $l \in \{0, \ldots, k\}$, $t \in \{0, 1, 2, \ldots\}$. We conjecture that (under 
mild additional assumptions that enforce an ``irreducible'' type property on the system) the actual and virtual systems have similar \emph{unconditional}  probabilities and expected costs. The conjecture is shown in simulation to hold for the robot example.  The resulting algorithm is simple, online, does not know the probability distribution, but achieves time average reward that is close to that of the best common sense heuristic that is fine tuned with knowledge of the distribution.

\subsection{Terminology} 

Let $\mathbb{N}=\{1, 2, 3, \ldots\}$. Fix $n,m\in\mathbb{N}$. The $L_2$ norm for  $x\in\mathbb{R}^n$ is $\norm{x}=\sqrt{\sum_{i=1}^nx_i^2}$. The $L_1$ norm is $\norm{x}_1=\sum_{i=1}^n|x_i|$.  

A \emph{measurable space} is a pair $(D,\script{G})$ where $D$ is a nonempty set and $\script{G}$ is a sigma algebra on $D$.  Unless otherwise stated:  If $(D_i, \script{G}_i)$ are measurable spaces for $i \in \{1, \ldots, m\}$, 
the sigma algebra for $\times_{i=1}^n D_i$ is assumed to be the standard product space sigma algebra; If $D$ is a Borel measurable subset of $\mathbb{R}^n$ then its sigma algebra is assumed to be $\script{B}(D)$, the standard Borel sigma algebra on $D$. 

A \emph{Borel space} is a measurable space that is isomorphic to $(C, \mathcal{B}(C))$ for some  $C\in\script{B}([0,1])$.   
An example Borel space is $(D,\script{B}(D))$ for some nonempty Borel measurable set $D \subseteq\mathbb{R}^n$. Another example is $(D, Pow(D))$ where $D$ is any nonempty finite or countably infinite set and $Pow(D)$ is the set of all subsets of $D$.  It is known that the finite or countably infinite product of Borel spaces is again a Borel space. 

Given a probability space $(\Omega, \script{F}, P)$ and some measurable space 
$(D,\script{G})$, a \emph{random element} is a  function $X:\Omega\rightarrow D$ that is measurable with respect to  input space $(\Omega, \script{F})$ and  output space $(D, \script{G})$, that is, $X^{-1}(B)\in\script{F}$ for all $B \in \script{G}$. A \emph{random variable} is a random element with output space $(D,\script{B}(D))$ for some  $D \in \script{B}(\mathbb{R})$.  We say $U\sim \mbox{Unif}[0,1]$ to indicate $U:\Omega\rightarrow [0,1]$ is a random variable  that is uniformly distributed over $[0,1]$.


 

\section{System model} 


Fix integers $n>0, k\geq 0$. The basic states are $\script{S}=\{1, \ldots, n\}$.   Let $c_{max}$ be a finite bound on the magnitude of all costs, so 
$$|C_l(t)|\leq c_{max} \quad \forall l \in \{0, \ldots, k\}, t \in \{0, 1, 2, \ldots\}$$
For each $i\in\script{S}$, let $(\script{A}_i, \script{G}_i)$ be a Borel space that is used for the actions $A_i(t)$ when $S(t)=i$.  Let $(\script{W}, \script{G}_W)$ be a measurable space (not necessarily a Borel space) 
that is used for the random elements $W(t)$.  
Define bounded and measurable \emph{cost functions} and \emph{transition probability functions} of the form 
\begin{align*}
&c_{i,l}:\script{W}\times\script{A}_i\rightarrow [-c_{max},c_{max}]\\
&p_{i,j}:\script{W}\times\script{A}_i\rightarrow [0,1]
\end{align*}
for $i,j \in \script{S}$ and $l \in \{0, \ldots, k\}$. The $p_{i,j}$ functions satisfy 
$$ \sum_{j=1}^np_{i,j}(w,a)=1 \quad \forall i \in \script{S},w \in \script{W}, a \in \script{A}_i$$

\subsection{Probability space} 

The probability space is $(\Omega, \script{F}, P)$.     Let $\{W(t)\}_{t=0}^{\infty}$, $\{U(t)\}_{t=0}^{\infty}$, $\{V(t)\}_{t=-1}^{\infty}$ be mutually independent, where 

\begin{itemize} 
\item $\{W(t)\}_{t=0}^{\infty}$ are i.i.d. 
random elements that take values in $\script{W}$. 

\item $\{U(t)\}_{t=0}^{\infty}$ are i.i.d. $\mbox{Unif}[0,1]$ (generated in software at the controller to enable  randomized actions).  

\item  $\{V(t)\}_{t=-1}^{\infty}$ are i.i.d. $\mbox{Unif}[0,1]$ (generated by ``nature'' to determine state transitions from  $S(t)$ to $S(t+1)$). 
\end{itemize} 

Let $\{S(t)\}_{t=0}^{\infty}$ be the sequence of basic states. 
Fix an initial state $s_0\in\script{S}$ and assume $S(0)=s_0$ surely. 
Define $H(0)=0$ and for $t \in \{1, 2, 3, \ldots\}$ define the \emph{history} $H(t)$ by
\begin{equation} \label{eq:history}
H(t)=(S(0), ..., S(t-1), W(0), ...,W(t-1), U(0), \ldots, U(t-1))
\end{equation} 
Define $D_H(t)$ as the set of possible values of $H(t)$ (so $D_H(0)=\{0\}$;  $D_H(t)= \script{S}^{t} \times \script{W}^{t}\times [0,1]^{t}$ for $t\in\{1, 2, 3, \ldots\}$).   Let $\script{G}_H(t)$ be the product sigma algebra on $D_H(t)$.

\subsection{Control policies} \label{section:control}

A \emph{causal and measurable policy} is a sequence of measurable functions $\{\alpha_{t}\}_{t=0}^{\infty}$ of the type:
\begin{align*}
&\alpha_{t}:\script{W}\times[0,1]\times D_H(t)\rightarrow \script{A}
\end{align*}
where $\script{A}=\times_{i=1}^n\script{A}_i$ and $\alpha_t=(\alpha_{t,1}, \ldots, \alpha_{t,n})$. 
These specify $A(t)=(A_1(t), \ldots, A_n(t))$ for each $t \in \{0, 1, 2, \ldots\}$ by 
\begin{equation} \label{eq:general-action} 
A_i(t) = \alpha_{t,i}(W(t), U(t), H(t)) \quad \forall i \in \script{S} 
\end{equation} 
One can view $U(t)$  as a random number generated by software at the controller at the start of each slot $t$.  The general structure \eqref{eq:general-action} allows $A_i(t)$ to be based on $U(0), \ldots, U(t)$. 
A causal and measurable policy is said to be \emph{memoryless} if it does not depend on $H(t)$, so 
$A(t) = \hat{\alpha}_{t}(W(t),U(t))$ 
for some measurable functions $\hat{\alpha}_{t}:\script{W}\times [0,1]\rightarrow\script{A}$ for $t \in \{0, 1, 2, \ldots\}$.

\subsection{State transitions} 

Recall that $S(0)=s_0$ surely. Define $\script{P}$ as the probability simplex on $\script{S}$: 
$$ \script{P} = \mbox{$\{(q_1, \ldots, q_n) : \sum_{i=1}^nq_i=1, q_i\geq 0\quad \forall i \in \{1, \ldots, n\}\}$}$$
Define a measurable function $\psi:\script{P}\times [0,1]\rightarrow \script{S}$  that takes a probability mass function $q\in \script{P}$ and a random seed $V(t) \sim \mbox{Unif}[0,1]$ and chooses a random state $\psi(q,V(t)) \in \script{S}$ according to mass function $q$.\footnote{One can use $\psi(q,r)=1$ if $r \in [0,q_1)$, $\psi(q,r)=2$ if $r \in [q_1, q_1+q_2)$, and so on.}
Define 
\begin{align}
S(t+1) = \sum_{i\in\script{S}} 1_{\{S(t)=i\}} \psi\left([p_{i,1}(W(t),A(t)), \ldots, p_{i,n}(W(t),A(t))], V(t)\right)\label{eq:s-def}
\end{align} 
Then $S(t+1)$ is a measurable function of $(S(t), W(t),A(t), V(t))$. This implies $S(t+1)$ is conditionally independent of $H(t)$ given $(S(t), W(t), A(t), V(t))$.  Therefore, the Markov property holds and  transition probabilities indeed satisfy \eqref{eq:pij}.  Since a measurable function of a random element is again a random element, by induction it can be shown that any causal and measurable control policy gives rise to valid random elements $\{S(t)\}_{t=0}^{\infty}$, $\{H(t)\}_{t=0}^{\infty}$, $\{A(t)\}_{t=0}^{\infty}$,  $\{C_l(t)\}_{t=0}^{\infty}$.\footnote{The random variable $V(-1)$ is never used on the system. Its existence formally allows a representation theorem in Appendix C.}

\subsection{Communicating classes} 

A $n \times n$ matrix $P=(P_{i,j})$ is a \emph{transition probability matrix} if it has nonnegative entries and each row sums to 1.   It is well known that such a matrix has at least one $\pi \in \script{P}$ that satisfies the  \emph{global balance equation} 
\begin{equation} \label{eq:GBE}
\pi_j = \sum_{i \in \script{S}} \pi_i P_{i,j} \quad \forall j \in \script{S}
\end{equation}
The matrix $P$  is \emph{irreducible}  if for all  $i,j \in \script{S}$, there is a positive integer $m$ and a sequence of states $s_0, s_1, \ldots, s_m$ in $\script{S}$ such that $P_{s_a,s_{a+1}}>0$ for all $a \in \{0, 1, ..., m-1\}$. 
It is well known that: (i) 
If the $n \times n$ transition probability matrix $P$ is irreducible then the global balance equation has a 
unique solution $\pi \in \script{P}$, and this unique solution has all positive entries; (ii) Regardless of irreducibility, $P$ 
has a unique collection of disjoint communicating classes $\script{S}^{(1)},\ldots, \script{S}^{(c)}$ (for some $c \in \{1, \ldots, n\}$), where each $\script{S}^{(j)}$ is a nonempty subset of $\script{S}$ and the Markov chain is irreducible when restricted to states in $\script{S}^{(j)}$. 

\section{Optimality} 

\subsection{The sets $\Gamma_i$}

We describe the set of all possible expectations of cost and transition probability for a given slot. Define 
$$\script{I} = \{(l,j) : l \in \{0, 1, ..., k\}, j \in \script{S}\}$$
Fix $i \in \script{S}$. 
Let $\script{C}_i$ be the collection of all measurable functions $\alpha_i:\script{W}\times[0,1]\rightarrow\script{A}_i$. Define  $g_i:\script{W}\times \script{A}\rightarrow \mathbb{R}^{|\script{I}|}$ by 
\begin{equation} \label{eq:g}
g_i(w,a) = (c_{i,l}(w,a); p_{i,j}(w,a))_{(l,j)\in\script{I}} \quad \forall w\in\script{W}, a \in \script{A}_i
\end{equation} 
Since  $g_i(W(0),A_i(0))$ is a bounded random vector, it has finite expectations. For notational simplicity, define $W=W(0)$, $U=U(0)$. Recall that $W$ and $U$ are independent and $U\sim \mbox{Unif}[0,1]$.  Define $\Gamma_i \subseteq\mathbb{R}^{|\script{I}|}$ for $i \in \script{S}$ by
\begin{equation} \label{eq:gamma-def} 
\Gamma_i = \{\expect{g_i(W, \alpha_i(W,U))} : \alpha_i \in \script{C}_i\}
\end{equation} 
Define $\overline{\Gamma}_i$ as the closure of set $\Gamma_i$. 
It can be shown that  $\overline{\Gamma}_i$ is compact and convex (see Appendix B). 

\subsection{Optimal cost} 

Consider the following deterministic optimization problem with decision variables $\pi_i, c_{i,l}, p_{i,j}$ for $i,j\in\script{S}$ and $l \in \{0, \ldots, k\}$.
\begin{align}
\mbox{Minimize:} &\quad \sum_{i \in \script{S}} \pi_i c_{i,0} \label{eq:det1} \\
\mbox{Subject to:} &\quad \pi_j = \sum_{i\in \script{S}} \pi_i p_{i,j} \quad \forall j \in \script{S} \label{eq:det2}\\
&\quad \sum_{i\in \script{S}} \pi_i c_{i,l}\leq 0 \quad \forall l \in \{1, \ldots, k\} \label{eq:det3}\\
&\quad (\pi_1, \ldots, \pi_n) \in \script{P} \label{eq:det4}\\
&\quad ((c_{i,l}); (p_{i,j})) \in \overline{\Gamma}_i \quad \forall i \in \script{S}\label{eq:det5}
\end{align}
This problem \eqref{eq:det1}-\eqref{eq:det5} is nonconvex due to  the multiplication of variables in \eqref{eq:det1}, \eqref{eq:det2}, \eqref{eq:det3}.  The problem is said to be \emph{feasible} if it is possible to satisfy the constraints 
\eqref{eq:det2}-\eqref{eq:det5}. Assuming feasibility, define $c_0^*$ as the infimum objective value in \eqref{eq:det1} subject to the constraints \eqref{eq:det2}-\eqref{eq:det5}. The next three results of this subsection are proven in Appendix B.

\begin{lem}\label{lem:exist} If problem \eqref{eq:det1}-\eqref{eq:det5} is feasible then an optimal solution $((\pi_i); (c_{i,l}); (p_{i,j}))$ for this problem exists for which
\begin{equation} \label{eq:equal-c0}
c_0^*=\sum_{i\in\script{S}} \pi_ic_{i,0}
\end{equation} 
and with $(p_{i,j})$ being a transition probability matrix with a communicating class 
$\script{S}'\subseteq\script{S}$ such that: (i) $\pi_i=0$ if $i \notin \script{S}'$; (ii) submatrix
$(p_{i,j})_{i,j\in\script{S}'}$ is irreducible;  (iii) $(\pi_i)_{i\in\script{S}'}$ is the unique solution to the stationary equations \eqref{eq:det2} over states in $\script{S}'$. 
\end{lem}


 \begin{thm}\label{thm:converse} (Converse) Suppose the stochastic problem \eqref{eq:p1}-\eqref{eq:p3} is feasible under a given initial state $s_0\in\script{S}$. Then the deterministic problem \eqref{eq:det1}-\eqref{eq:det5} is also feasible. Further, if  $S(0)=s_0$ surely and if $\{A(t)\}_{t=0}^{\infty}$ are causal and measurable control actions of the form \eqref{eq:general-action} that ensure the constraints \eqref{eq:p2} are satisfied 
 almost surely, the resulting $C_0(t)$ process satisfies 
\begin{align}
&\liminf_{T\rightarrow\infty} \frac{1}{T}\sum_{t=0}^{T-1} C_0(t) \geq c_0^* \quad \mbox{(almost surely)} \label{eq:converse}\\
&\liminf_{T\rightarrow\infty}  \frac{1}{T}\sum_{t=0}^{T-1} \expect{C_0(t)} \geq c_0^* \label{eq:fatou}
\end{align}
 \end{thm} 
 
 \begin{thm}  \label{thm:achievability} (Achievability) Suppose set $\Gamma_i\subseteq\mathbb{R}^{|\script{I}|}$ is closed for each $i \in \script{S}$.  
 If the deterministic problem \eqref{eq:det1}-\eqref{eq:det5} is feasible, then there is an initial
 state $s_0 \in \script{S}$  for which  the stochastic problem \eqref{eq:p1}-\eqref{eq:p3} is feasible. Further, there exists a measurable function $\alpha:\script{W}\times[0,1]\rightarrow\script{A}$ such that when $S(0)=s_0$ surely, using the memoryless actions $A^*(t)=\alpha(W(t), U(t))$ for $t \in\{0, 1, 2, \ldots\}$ yields costs  that almost surely satisfy 
 \begin{align*}
&\lim_{T\rightarrow\infty} \frac{1}{T}\sum_{t=0}^{T-1}C_0^*(t) = c_0^* \\
 &\lim_{T\rightarrow\infty} \frac{1}{T}\sum_{t=0}^{T-1}C_l^*(t) \leq 0 \quad \forall l \in \{1, \ldots, k\} 
 \end{align*}
 and these limiting time averages are almost surely the same as the corresponding limiting time average expectations. 
 \end{thm}

The assumption that sets $\Gamma_i$ are compact is crucial for
Theorem \ref{thm:achievability} to be stated in terms of a single memoryless
action function $\alpha:\script{W}\times[0,1]\rightarrow\script{A}$.
A counter-intuitive example by Blackwell in  \cite{blackwell-memoryless} 
shows that memoryless actions are insufficient for systems defined on Borel
sets with nonBorel projections,  see also \cite{savage-gamble-selection}. 

Theorem \ref{thm:achievability} requires a proper choice of initial state $s_0 \in \script{S}$. Specifically, 
$s_0$ should be chosen as any state in the desirable communicating class $\script{S}'$ described in  
Lemma \ref{lem:exist}. The 
 transition probability matrix $(p_{i,j})$ of Lemma \ref{lem:exist} can also have \emph{undesirable} communicating classes
over which the global balance equations are satisfied but optimal cost is not achieved.  Our simulations of the robot example in Section \ref{section:simulation} show these undesirable communicating classes can arise in certain cases when it is optimal to restrict the robot to a small subset of locations (so an ``irreducible'' type property fails).  As shown in Section \ref{section:simulation}, the \emph{virtual system} is unaffected by this 
issue. However, the \emph{actual system} can get trapped in an undesirable communicating class that prevents its unconditional probabilities from matching those of the virtual system (even though the two have the same conditional distributions). Fortunately, this situation is easy to detect and a simple online fix is discussed in Section \ref{section:simulation} that ``kicks'' the actual system out of trapping states and thereby maintains desirable unconditional behavior.


\subsection{Lagrange multipliers}

While problem \eqref{eq:det1}-\eqref{eq:det5} is nonconvex, it can be shown that 
the set of all vectors of the form 
\begin{equation}
(\pi; (\pi_i c_{i,l})_{i\in\script{S}, l \in\{0, ..., k\}}; (\pi_i p_{ij})_{i,j\in\script{S}}) \label{eq:convex-points}
\end{equation} 
for some values $\pi=(\pi_1, \ldots, \pi_n), c_{i,l}, p_{i,j}$ that satisfy 
\begin{align}
&\pi \in \script{P} \label{eq:convex-points2}\\
&((c_{i,l});(p_{i,j}))\in\overline{\Gamma}_i \quad \forall i \in \script{S} \label{eq:convex-points3}
\end{align} 
is a convex set. 
Therefore, the following additional \emph{Lagrange multiplier assumption} is mild. 

\begin{assumption} \label{assumption:LM} (Lagrange multipliers) The problem \eqref{eq:det1}-\eqref{eq:det5} is feasible with optimal objective $c_0^*$, and there are real numbers $\mu_l\geq 0$ for $l \in \{0, \ldots, k\}$ and  $\lambda_j\in \mathbb{R}$ for $j \in \script{S}$ such that 
$$ \sum_{i\in \script{S}} \pi_ic_{i,0} + \sum_{j\in\script{S}}\lambda_j\left(\pi_j - \sum_{i \in \script{S}} \pi_ip_{i,j}\right) + \sum_{l=1}^k \mu_l \left(\sum_{i\in\script{S}} \pi_ic_{i,l}\right)\geq c_0^*$$
for all $(\pi; (c_{i,l});(p_{i,j}))$ that satisfy \eqref{eq:convex-points2}-\eqref{eq:convex-points3}.
\end{assumption} 

Assumption \ref{assumption:LM} (together with Lemma \ref{lem:in-gamma} in Appendix B) implies that for any random element $W:\Omega\rightarrow\script{W}$ with the same distribution as $W(0)$, any random element $\pi:\Omega\rightarrow\script{P}$ that is independent of $W$, and 
 any random element $A:\Omega\rightarrow\script{A}$ with arbitrary dependence on $(\pi, W)$ we have 
\begin{align}
\sum_{i\in\script{S}} \expect{\pi_i c_{i,0}(W,A_i)}
+ \sum_{j\in\script{S}} \lambda_j\expect{\pi_j-\sum_{i\in\script{S}}\pi_ip_{i,j}(W,A_i)} 
+ \sum_{l=1}^k\mu_l\sum_{i\in\script{S}} \expect{\pi_i c_{i,l}(W,A_i)} \geq c_0^* \label{eq:LM-consequence}
\end{align}

\section{Algorithm development}

\subsection{K-L divergence} 

Recall that $\script{P}$ is the simplex of all probability mass functions on $\script{S}$. Define $\script{P}_o\subseteq\script{P}$ by 
$$ \script{P}_o=\{(\pi_1, \ldots, \pi_n) : \sum_{i=1}^n\pi_i=1,\pi_i>0 \quad \forall i \}$$
Define $D:\script{P}\times \script{P}_o\rightarrow\mathbb{R}$ as the Kullback-Leibler (K-L) divergence function:
$$ D(p;q) = \sum_{i=1}^n p_i\log(p_i/q_i) \quad \forall p\in\script{P}, q \in \script{P}_o$$
where $\log(\cdot)$ denotes the natural logarithm. 
It is well known that $D(\cdot)$ is a nonnegative function that satisfies
\begin{align}
&D(p; (1/n, \ldots, 1/n)) \leq  \log(n) \quad \forall p \in \script{P}\label{eq:logn}\\
&D(p;q) \geq \frac{1}{2}\norm{p-q}_1^2 \quad \forall p \in \script{P}, q \in \script{P}_o \label{eq:pinsker}
\end{align}
where $\norm{x}_1=\sum_{i=1}^n|x_i|$. Inequality \eqref{eq:pinsker} is the \emph{Pinsker inequality}.  The following \emph{pushback lemma} is standard \cite{tseng-accelerated}\cite{xiaohan-selm}\cite{SGD-robust}:

\begin{lem} \label{lem:pushback} (Pushback)  Fix $q\in \script{P}_o$, $\alpha\geq0$. Let 
$A\subseteq\script{P}$ be a convex set and  $f:A\rightarrow\mathbb{R}$ a convex function.  Suppose $p^{opt}$ solves
\begin{align}
\mbox{Minimize:} \quad & f(p) + \alpha D(p;q) \label{eq:push1}\\
\mbox{Subject to:} \quad & p \in A \label{eq:push2}
\end{align}
If $p^{opt} \in \script{P}_o$ then:
\begin{equation} \label{eq:pushback} 
f(p^{opt}) + \alpha D(p^{opt};q) \leq f(p) + \alpha D(p;q) - \alpha D(p,p^{opt}) \quad \forall p \in A
\end{equation} 
\end{lem}

\subsection{Time averaged problem} 

The algorithm observes the i.i.d. $\{W(t)\}_{t=0}^{\infty}$ and makes contingency actions $A_i(t) \in \script{A}_i$ for all $i \in \script{S}$ and $t \in \{0, 1, 2, \ldots\}$. 
It introduces a decision vector $\pi(t)= (\pi_1(t), \ldots, \pi_n(t)) \in \script{S}$ that intuitively represents the state probability on a virtual system where we can choose any state in $\script{S}$ that we like on each new slot.  The goal is to make these decisions over time to solve
\begin{align}
\mbox{Minimize:} &\quad \overline{\sum_{i \in \script{S}} \pi_i(t)c_{i,0}(W(t),A_i(t))} \label{eq:r1}\\
\mbox{Subject to:} &\quad \overline{\pi_j(t) - \sum_{i\in\script{S}}\pi_i(t)p_{i,j}(W(t),A_i(t))} = 0 \quad \forall j \in \script{S}\label{eq:r2}\\
&\quad \overline{\sum_{i \in \script{S}} \pi_i(t)c_{i,l}(W(t),A_i(t))}\leq0 \quad \forall l \in \{1, \ldots, k\}\label{eq:r3} \\
&\quad \pi(t) \in \script{P} \quad \forall t \in \{0, 1, 2, ...\} \label{eq:r4}\\
&\quad A_i(t) \in \script{A}_i  \quad \forall i \in \script{S}, t \in \{0,1,2,...\}\label{eq:r5}\\
&\quad \mbox{$\pi(t)$ and $W(t)$ are independent for each $t \in \{0, 1, 2, ...\}$} \label{eq:r6}
\end{align}
where the overbar notation in \eqref{eq:r1}, \eqref{eq:r2}, \eqref{eq:r3} denotes a limiting time average, for example 
$$ \overline{\sum_{i \in \script{S}}\pi_i(t)c_{i,l}(W(t),A_i(t))}  = \limsup_{T\rightarrow\infty} \frac{1}{T}\sum_{t=0}^{T-1}\sum_{i \in \script{S}} \pi_i(t)c_{i,l}(W(t),A_i(t))$$
The above problem is a time averaged version of the deterministic problem 
\eqref{eq:det1}-\eqref{eq:det5}.  Specifically, the constraints \eqref{eq:r2}-\eqref{eq:r4} directly relate to the deterministic constraints \eqref{eq:det2}-\eqref{eq:det4}. The constraints \eqref{eq:r2} shall be called the \emph{global balance constraints}. The global balance constraints play a crucial role in the learning process. 
The constraints \eqref{eq:r5}, \eqref{eq:r6} together correspond to the deterministic constraint \eqref{eq:det5}. Constraint \eqref{eq:r6} is a nonstandard and subtle constraint that prevents $\pi(t)$ from being influenced by the realization of $W(t)$. Without \eqref{eq:r6}, the  decisions for $\pi(t)$ would be biased towards those states for which there is a favorable realization of $W(t)$, which would not correspond to any solution that could be achieved on the actual system. Our method of enforcing \eqref{eq:r6} ensures an algorithm complexity that depends only on the size of the finite set $\script{S}$ (which is $n$), rather than the size of the (possibly infinite) set $\script{S}\times\script{W}$.

\subsection{Virtual queues}

For $t \in \{0, 1, 2, \ldots\}$ define the $n \times 1$ vector $G_0(t)$ by 
$$G_0(t) = (c_{1,0}(W(t), A_1(t)), \ldots, c_{n,0}(W(t),A_n(t)))$$
Define the 
 $n \times n$ matrix $Y(t)=(Y_{i,j}(t))$ and the $n \times k$ matrix $G(t)=(G_{i,l}(t))$ by 
 \begin{align*} 
Y_{i,j}(t)&= 1_{\{i=j\}} - p_{i,j}(W(t),A_i(t)) \quad \forall i,j\in\script{S}\\
G_{i,l}(t) &= c_{i,l}(W(t), A_i(t))) \quad \forall i\in\script{S},l\in\{1, \ldots, k\}
\end{align*}
where $1_{\{i=j\}}$ is 1 if $i=j$, and 0 else. 
For initialization, define $G_0(-1)=-(c_{max}, \ldots, c_{max})$, $Y(-1)=0$, $G(-1)=0$.

Using the virtual queue technique of \cite{sno-text} 
to enforce the constraints \eqref{eq:r2}, \eqref{eq:r3}, define processes $Q_j(t)$ and $Z_l(t)$ with initial conditions $Q_j(0)=0$, $Z_l(0)=0$ and update equation for $t\in \{0, 1, 2, \ldots\}$  given by 
\begin{align}
Q_j(t+1)&= Q_j(t) + \pi(t)^{\top}Y(t-1)y_j  \quad \forall j \in \script{S} \label{eq:q-update}\\
Z_l(t+1) &= \max\left[Z_l(t) + \pi(t)^{\top}G(t-1)g_l , 0\right]  \quad \forall l \in \{1, \ldots, k\} \label{eq:z-update}
\end{align}
where 
$\pi^{\top}(t) = (\pi_1(t), \ldots, \pi_n(t))$ is a row vector; $y_j\in \mathbb{R}^n$ is the unit vector that is 1 in entry $j$ and $0$ in all other entries; $g_l \in \mathbb{R}^k$ is the unit vector that is 1 in entry $j$ and $0$ in all other entries.  In particular, $Y(t-1)y_j$ selects the $j$th column of matrix $Y(t-1)$, while $G(t-1)g_l$ selects the $l$th column of matrix $G(t-1)$. 

\begin{lem} Consider the iterations \eqref{eq:q-update}, \eqref{eq:z-update} with any decisions $\pi(t)\in\script{P}$ and $A_i(t) \in \script{A}_i$ for $i \in \script{S}$ and $t\in \{0, 1, 2, \ldots\}$.  For all positive integers $T$ we have
\begin{align}
&\left|\frac{1}{T}\sum_{t=0}^{T-1}\pi(t)^{\top}Y(t)y_j\right| \leq \left|\frac{Q_j(T+1)}{T}\right|+ \frac{1}{T}\sum_{t=1}^{T}\norm{\pi(t)-\pi(t-1)}_1 \quad \forall j \in \script{S}\label{eq:vqq}\\
&\frac{1}{T}\sum_{t=0}^{T-1}\pi(t)^{\top}G(t)g_l \leq \frac{Z_l(T+1)}{T} + \frac{c_{max}}{T}\sum_{t=1}^T\norm{\pi(t)-\pi(t-1)}_1 \quad \forall l \in \{1, \ldots, k\} \label{eq:vqz}
\end{align}
\end{lem} 

\begin{proof} 
Fix $j \in \script{S}$. Since $Q_j(0)=0$, for all positive integers $T$ we have 
\begin{align*}
Q_j(T+1) &=  \sum_{t=0}^{T}(Q_j(t+1)-Q_j(t))\\
&\overset{(a)}{=} \sum_{t=0}^{T} \pi(t)^{\top}Y(t-1)y_j \\
&\overset{(b)}{=} \sum_{t=1}^T\pi(t)^{\top}Y(t-1)y_j \\
&=\sum_{t=1}^T\pi(t-1)^{\top}Y(t-1)y_j + \sum_{t=1}^T(\pi(t)-\pi(t-1))^{\top}Y(t-1)y_j
\end{align*}
where (a) holds by \eqref{eq:q-update}; (b) holds because $Y(-1)=0$. 
Thus 
\begin{align*}
\left|\sum_{t=1}^{T}\pi(t-1)^{\top}Y(t-1)y_j\right| \leq  \left|Q_j(T+1)\right| +   \sum_{t=1}^T\norm{\pi(t)-\pi(t-1)}_1 
\end{align*}
which holds because the largest magnitude of any component of $Y(t-1)y_j$ is 1. Dividing both sides by $T$ proves \eqref{eq:vqq}. 

Fix $l \in \{1, \ldots, k\}$. The update \eqref{eq:z-update} implies 
$$ Z_l(t+1) \geq Z_l(t) + \pi(t)^{\top}G(t-1)g_l  \quad \forall t \in \{0, 1, 2, \ldots\}$$
Summing over $t \in \{0, \ldots, T\}$ for some positive integer $T$ and using $Z_l(0)=0$ gives 
\begin{align*}
Z_l(T+1) &\geq \sum_{t=0}^T \pi(t)^{\top}G(t-1)g_l \\
&\overset{(a)}{=}\sum_{t=1}^T\pi(t)^{\top}G(t-1)g_l \\
&\geq \sum_{t=1}^T\pi(t-1)^{\top}G(t-1)g_l - c_{max}\sum_{t=1}^T\norm{\pi(t)-\pi(t-1)}_1 
\end{align*}
where (a) holds because $G_l(-1)=0$; the final inequality holds because $c_{max}$ bounds the largest magnitude of any component of vector $G(t-1)g_l$. 
Dividing this by $T$ yields \eqref{eq:vqz}. 
\end{proof}

\subsection{Lyapunov drift}

For $t \in \{0, 1, 2, \ldots\}$ define 
\begin{align*}
Q(t) &= (Q_1(t), \ldots, Q_n(t)) \\
Z(t) &= (Z_1(t), \ldots, Z_k(t)) \\
J(t) &= (Q(t);Z(t)) 
\end{align*} 
$$L(t) = \frac{1}{2}\norm{J(t)}^2=\frac{1}{2}\sum_{j=1}^n Q_j(t)^2+ \frac{1}{2}\sum_{l=1}^kZ_l(t)^2 $$
Define $\Delta(t) = L(t+1)-L(t)$.

\begin{lem} \label{lem:Delta} For all $t\in\{0, 1, 2, \ldots\}$ and all choices of the decision variables we have 
\begin{align*}
\Delta(t) &\leq b_1 + \pi(t)^{\top}Y(t-1)Q(t) +\pi(t)^{\top}G(t-1)Z(t) \end{align*}
where 
$b_1 = \frac{n+kc_{max}^2}{2}$. 
\end{lem} 

\begin{proof} 
The result follows by 
squaring \eqref{eq:q-update} and \eqref{eq:z-update} and using $\max[z,0]^2\leq z^2$ for $z \in \mathbb{R}$, which is a standard Lyapunov optimization technique (see, for example, \cite{sno-text}). 
\end{proof} 

Fix $V>0, \alpha>0$ as parameters to be sized later.  Define the following \emph{drift-plus-penalty} expression
$$ \Delta(t) + V\pi(t)^{\top}G_0(t-1) + \alpha D(\pi(t);\pi(t-1))$$
Our proposed algorithm makes decisions that minimize an upper bound on this expression. Intuitively, $\Delta(t)$ can be viewed as a drift term whose inclusion helps to maintain small values of the virtual queues. The term $V\pi(t)^{\top}G_0(t-1)$ is a penalty (weighted by the $V$ parameter) whose inclusion helps to maintain small values of the objective in \eqref{eq:r1}. We would like the actual objective to be $\pi(t)G_0(t)$, but we shall choose $\pi(t) \in \script{P}$ based on knowledge of $G_0(t-1)$, rather than $G_0(t)$, to ensure $\pi(t)$ is independent of $W(t)$ (recall the independence requirement \eqref{eq:r6}). To reduce the error associated with using $G_0(t-1)$ instead of $G_0(t)$, the additional penalty $D(\pi(t); \pi(t-1))$ is included (weighted by the $\alpha$ parameter) to ensure $\pi(t)$ and $\pi(t-1)$ are close. 
From Lemma \ref{lem:Delta} the drift-plus-penalty expression has the following upper bound
\begin{align}
&\Delta(t) + V\pi(t)^{\top}G_0(t-1) + \alpha D(\pi(t);\pi(t-1))\nonumber\\
&\leq b_1   +\pi(t)^{\top}\left[V G_0(t-1)+ Y(t-1)Q(t)+  G(t-1)Z(t)\right] + \alpha D(\pi(t);\pi(t-1)) \label{eq:DPP}
\end{align}
The following algorithm uses a layered structure to make the right-hand-side of the above expression small. 

\subsection{Online algorithm} 

Initialize:
\begin{align*}
&Q_j(-1)=Y_{i,j}(-1)=0 \quad \forall i,j \in \script{S}\\
&Z_l(-1)=G_{i,l}(-1)=0 \quad \forall l \in \{1, \ldots, k\}, i\in\script{S}\\
&G_0(-1)=-(c_{max}, \ldots, c_{max})\\
&\pi(-1)=(1/n, 1/n, ...., 1/n) \in \script{P}
\end{align*}
On each slot $t \in \{0, 1, 2, \ldots\}$ do: 

\begin{itemize} 
\item Layer 1:  Ignore $W(t)$. Choose $\pi(t) \in \script{P}$ to minimize 
\begin{equation}\label{eq:layer1} 
\pi(t)^{\top}\left[VG_0(t-1) + Y(t-1)Q(t)+ G(t-1)Z(t) \right] + \alpha D(\pi(t);\pi(t-1))
\end{equation} 
treating $VG_0(t-1)$, $Y(t-1)Q(t)$, $G(t-1)Z(t)$, $\pi(t-1)$ as known constants. This has the following solution $\pi(t) \in \script{P}_o$:
$$ \pi_i(t) = \frac{\pi_i(t-1)\exp(-M_i(t)/\alpha)}{\sum_{j=1}^n\pi_j(t-1)\exp(-M_j(t)/\alpha)} \quad \forall i \in \script{S}$$
where $M_i(t)$ is defined for $i \in \script{S}$ by 
$$ M_i(t) = y_i^{\top}\left(VG_{0}(t-1) + Y(t-1)Q(t) + G(t-1)Z(t)\right)$$
where $y_i\in\mathbb{R}^n$ is the unit vector that is 1 in entry $i$ and $0$ in all other entries. 
\item Layer 2: Observe $W(t)$. For each $i \in \script{S}$ choose $A_i(t) \in \script{A}_i$ to minimize:\footnote{For simplicity we assume a minimizer $A_i(t)\in \script{A}_i$ exists. This holds when $\script{A}_i$ is a finite set. More generally we can use a $\delta$-approximation to the infimum, for some small $\delta>0$, and all performance theorems hold by increasing the $b$ constant in Lemma \ref{lem:key} by $\delta$ (so there is no significant change in the results). Of course, for general problems, the complexity of computing a $\delta$-approximation may depend on the size of the virtual queue weights $Z_l(t)$ and $Q_j(t)$.}
\begin{equation} \label{eq:layer2}
Vc_{i,0}(W(t),A_i(t)) + \sum_{l=1}^kZ_l(t)c_{i,l}(W(t),A_i(t))- \sum_{j=1}^nQ_j(t)p_{i,j}(W(t),A_i(t))
\end{equation} 
treating $W(t)$, $Z_l(t)$, $Q_j(t)$ as known constants. 
\item  Virtual queue update: Update virtual queues by \eqref{eq:q-update} and \eqref{eq:z-update}.
\item Actual system implementation: Observe the actual system state $S(t) \in \script{S}$. Apply action $A_{S(t)}$.
\end{itemize}

\subsection{Discussion} 

To analyze the algorithm, we ignore the actual system state dynamics $\{S(t)\}_{t=0}^{\infty}$.  Instead, the sequence $\{\pi(t), W(t), A(t)\}_{t=0}^{\infty}$ can be viewed as a virtual system that replaces the actual state $S(t)$ with a mass function $\pi(t) \in \script{P}$. 
For any $\epsilon>0$, we choose $\alpha = 1/\epsilon^2$ and $V=\rho/\epsilon$ for some constant $\rho>0$, such as $\rho=1$, and prove that our decision variables satisfy 
\begin{align}
&\frac{1}{T}\sum_{t=0}^{T-1}\expect{\pi(t)G_0(t)} \leq c_0^* + O(\epsilon) \label{eq:perf1}\\
&\frac{1}{T}\expect{\norm{J(T)}} \leq O(\epsilon)\label{eq:perf2}\\
&\expect{\left|\frac{1}{T}\sum_{t=0}^{T-1}\pi(t)^{\top}Y(t)y_j\right|} \leq O(\epsilon) \quad \forall j \in \script{S} \label{eq:perf3}\\
&\expect{\frac{1}{T}\sum_{t=0}^{T-1}\pi(t)^{\top}G(t)g_l} \leq O(\epsilon) \quad \forall l \in \{1, \ldots, k\}\label{eq:perf4}
\end{align}
so virtual time average expected cost is within $O(\epsilon)$ of the optimal $c_0^*$ value and 
all constraints have time average expectations within $O(\epsilon)$ of their desired values.  This produces an $O(\epsilon$)-optimal solution to \eqref{eq:r1}-\eqref{eq:r6}.  In particular, the virtual decisions $\pi(t)$ are very close to satisfying the global balance constraints needed for the actual system. 

How do costs and transition probabilities in the virtual system relate to those in the actual system?  Observe that the \emph{conditional} transition probabilities and \emph{conditional} costs on the virtual and actual systems are exactly the same for all time $t$, all $i,j\in\script{S}$, and all $l \in \{0, 1, \ldots, k\}$:
\begin{align*}
&P[S(t+1)=j|S(t)=i, H(t)]= \expect{p_{i,j}(W(t),A_i(t))|H(t)}\\
&\expect{C_l(t)|S(t)=i, H(t)} = \expect{c_{i,l}(W(t),A_i(t))|H(t)} 
\end{align*}
Thus, the time average expected costs on the actual system will be the same as the virtual system if both systems have the same expected fraction of time being in each state $i \in \script{S}$. 
Intuitively, the fact that both actual and virtual systems have the same conditional 
transition probabilities between states suggests they both spend similar fractions of time in each state.   

As a (nonrigorous) thought experiment, imagine a fixed matrix $P=(P_{i,j})$ that acts as a ``limiting average transition probability'' matrix that is common to both actual and virtual systems. The actual system \emph{naturally} satisfies a global balance constraint with respect to $P$ because the number of times visiting each state $i \in \script{S}$ is always within 1 of the number of times exiting. Meanwhile, the virtual system is \emph{carefully controlled}
to satisfy the global balance constraints with respect to $P$. If $P$ is irreducible then there is only one solution to the global balance constraints, so the steady state probabilities across the virtual and actual systems must be the same! 
We believe a theorem in this direction can be proven by imposing additional mild structure on the problem. In particular, we conjecture that forcing the system to revisit state 1 with some small probability $\delta>0$ on every slot creates an irreducibility property that induces a unique steady state that attracts both virtual and actual systems.   
Rather than analytically pursue this direction, we restrict the mathematical analysis in this paper to the virtual system. This analysis uses a novel layered approach to achieve the desired time averages under the constraint that $\pi(t)$ and $W(t)$ are independent. Then, we simulate the system on the robot example. Simulations reveal that behavior of the virtual and actual system is similar, with the exception of particular cases when irreducibility dramatically 
fails. Fortunately, the exceptional cases are easy to detect, and a simple redirect mode is introduced in Section \ref{section:simulation} that recovers desirable behavior. 

\subsection{Layered analysis} 

For the remainder of this paper we assume the sets $\Gamma_i$ are closed for all $i \in \script{S}$ (so $\Gamma_i=\overline{\Gamma}_i$) and the problem \eqref{eq:det1}-\eqref{eq:det5} is feasible with optimal solution $(\pi^*;(c_{i,l}^*); (p_{i,j}^*))$ where  $\pi^*\in\script{P}$ and 
\begin{equation} \label{eq:in-gamma-dude}
((c_{i,l}^*); (p_{i,j}^*))_{(l,j)\in\script{I}}\in\Gamma_i \quad \forall i \in \script{S}
\end{equation} 
For each $i \in \script{S}$, by definition of $\Gamma_i$, there is a function $\alpha_i \in \script{C}_i$ such that defining $A_i^*(t)=\alpha_i(W(t),U(t))$ yields
\begin{align}
&\expect{c_{i,l}(W(t), A_i^*(t))} = c_{i,l}^*\label{eq:Good1}\\
&\expect{p_{i,j}(W(t),A_i^*(t))} = p_{i,j}^*\label{eq:Good2}
\end{align}
for all $t \in \{0, 1, 2, \ldots\}$, all $l \in \{0, \ldots, k\}$, $j \in \script{S}$.

\begin{lem} \label{lem:key} For all $t \in \{0, 1, 2, \ldots\}$ we have 
\begin{align}
\expect{\Delta(t) + V\pi(t)^{\top}G_0(t-1) + \alpha D(\pi(t);\pi(t-1))} \leq b + Vc_0^*+ \alpha \expect{D(\pi^*;\pi(t-1))-\alpha D(\pi^*;\pi(t))} \label{eq:result-prev-lem} 
\end{align}
where $b = (3/2)(n+kc_{max}^2)$.
\end{lem} 

\begin{proof} 
Fix $t \in \{0, 1, 2, \ldots\}$. For each $i \in \script{S}$, the layer 2 decision chooses $A_i(t) \in \script{A}_i$ to minimize \eqref{eq:layer2} and so 
\begin{align*}
&Vc_{i,0}(W(t),A_i(t)) + \sum_{l=1}^kZ_l(t)c_{i,l}(W(t),A_i(t)) - \sum_{j=1}^nQ_j(t)p_{i,j}(W(t),A_i(t))\\
&\leq  Vc_{i,0}(W(t),A_i^*(t)) + \sum_{l=1}^kZ_l(t)c_{i,l}(W(t),A_i^*(t))-  \sum_{j=1}^nQ_j(t)p_{i,j}(W(t),A_i^*(t))
\end{align*}
where $A_i^*(t)$ is any other (possibly randomized) decision in $\script{A}_i$. 
Multiplying the above by $\pi_i^*$, summing over $i \in \{1, \ldots, n\}$, and adding $\sum_{j=1}^nQ_j(t)\pi_j^*$ to both sides gives 
\begin{align}
&\pi^{*\top}\left[VG_0(t) + G(t)Z(t) + Y(t)Q(t)\right]\nonumber\\
&\leq V\sum_{i=1}^n\pi_i^*c_{i,0}(W(t),\alpha_i^*(t)) + \sum_{l=1}^kZ_l(t)\sum_{i=1}^n\pi_i^*c_{i,l}(W(t),A_i^*(t)) + \sum_{j=1}^nQ_j(t)\left[\pi_j^*-\sum_{i=1}^n\pi_i^*p_{i,j}(W(t),A_i^*(t))\right] \label{eq:pre-expect} 
\end{align}
For $i\in\script{S}$, let $A_i^*(t) = \alpha_i(W(t),V)$ be the action that is independent of history $H(t)$ (and hence independent of the virtual queue values) that yields \eqref{eq:Good1}-\eqref{eq:Good2}.
Taking expectations of \eqref{eq:pre-expect} gives
\begin{align} 
&\pi^{*\top}\expect{VG_0(t) + G(t)Z(t) + Y(t)Q(t)}\nonumber \\
&\leq V\sum_{i=1}^n\pi_i^*c_{i,0}^*+ \sum_{l=1}^kZ_l(t)\sum_{i=1}^n\pi_i^*c_{i,l}^*+ \sum_{j=1}^nQ_j(t)\left[\pi_j^*-\sum_{i=1}^n\pi_i^*p_{i,j}^*\right] 
\end{align}
By definition of an optimal solution to \eqref{eq:det1}-\eqref{eq:det5} we have $c_{i,l}^*\leq 0$, 
$\pi_j^*=\sum_{i=1}^n\pi_i^*p_{i,j}^*$, and 
$c_0^*=\sum_{i=1}^n \pi_i^*c_{i,0}^*$ 
and so 
\begin{equation} \label{eq:inner-layer} 
\pi^{*\top}\expect{VG_0(t) + G(t)Z(t) + Y(t)Q(t)} \leq Vc_0^*
\end{equation} 
By definition of $G_0(-1)=-(c_{max}, \ldots, c_{max})$, inequality \eqref{eq:inner-layer} also holds for time $t=-1$. 

The layer 1 decision chooses $\pi(t)$ in the convex set $\script{P}$ to minimize $\alpha D(\pi(t);\pi(t-1))$ plus a linear function of $\pi(t)$, so by the pushback lemma (Lemma \ref{lem:pushback}): 
\begin{align}
&\pi(t)^{\top}\left[VG_0(t-1) + Y(t-1)Q(t) +  G(t-1)Z(t) \right] + \alpha D(\pi(t);\pi(t-1))\nonumber\\
&\leq \pi^{*\top}\left[ VG_0(t-1) + G(t-1)Z(t) + Y(t-1)Q(t)\right] + \alpha D(\pi^*;\pi(t-1))-\alpha D(\pi^*;\pi(t))\label{eq:sub44} 
\end{align}
Define $b_2 = kc_{max}^2 + n$. 
Since the $|Z_l(t)-Z_l(t-1)|\leq c_{max}$ and $|Q_j(t)-Q_j(t-1)|\leq 1$ we obtain from \eqref{eq:sub44}
\begin{align*}
&\expect{\pi(t)^{\top}\left[VG_0(t-1) + Y(t-1)Q(t) +  G(t-1)Z(t) \right] + \alpha D(\pi(t);\pi(t-1))}\\
&\leq b_2 + \pi^{*\top}\expect{VG_0(t-1)+G(t-1)Z(t-1) + Y(t-1)Q(t-1)} + \alpha \expect{D(\pi^*;\pi(t-1))-\alpha D(\pi^*;\pi(t))}\\
&\overset{(a)}{\leq} b_2 + Vc_0^* + \alpha \expect{D(\pi^*;\pi(t-1))-\alpha D(\pi^*;\pi(t))}
\end{align*}
where (a) uses the fact that \eqref{eq:inner-layer} holds for all $t \in \{-1, 0, 1, 2, \ldots\}$. 
Adding $b_1$ to both sides, defining $b=b_1+b_2$, and substituting into inequality \eqref{eq:DPP} proves the result. 
\end{proof} 

\

\begin{thm}\label{thm:performance} (Performance) Assuming the problem is feasible, we have for all positive integers $T$:
\begin{align} 
\frac{1}{T}\sum_{t=0}^{T-1}\expect{\pi(t)G_0(t)} - c_0^* \leq  \frac{b(1+\frac{1}{T})}{V} + \frac{Vc_{max}^2(1+\frac{1}{T})}{2\alpha}+ \frac{c_0^*+c_{max}}{T}  + \frac{\alpha\log(n)}{VT} \label{eq:costs}
\end{align}
where $b=(3/2)(n+kc_{max}^2)$. 
If the Lagrange multiplier assumption (Assumption \ref{assumption:LM}) 
holds then for all positive integers $T$: 
\begin{align}
&\expect{\norm{J(T)}} \leq V\norm{(\lambda;\mu)} +\sqrt{V^2\norm{(\lambda;\mu)}^2 + 2bT + 2\alpha \log(n) + \frac{V^2d^2T}{\alpha}}\label{eq:queuebound} \\
&\expect{\left|\frac{1}{T}\sum_{t=0}^{T-1}\pi(t)^{\top}Y(t)y_j\right|} \leq \frac{\norm{J(T+1)}}{T} + \sqrt{(\frac{4b}{\alpha}+ \frac{4V^2d^2}{\alpha^2})(1+\frac{1}{T}) + \frac{4\log(n)}{T} + \frac{2V^2\norm{(\lambda;\mu)}^2}{\alpha T}} \label{eq:Dav} \\
&\expect{\frac{1}{T}\sum_{t=0}^{T-1}\pi(t)^{\top}G(t)g_l}\leq \frac{\norm{J(T+1)}}{T} + c_{max}\sqrt{(\frac{4b}{\alpha}+ \frac{4V^2d^2}{\alpha^2})(1+\frac{1}{T}) + \frac{4\log(n)}{T} + \frac{2V^2\norm{(\lambda;\mu)}^2}{\alpha T}}\label{eq:Dav2}
\end{align}
where $J(t)=(Q(t);Z(t))$, 
$\norm{(\lambda;\mu)}^2 = \sum_{j=1}^n\lambda_j^2+\sum_{l=1}^k\mu_l^2$, and 
$$d=c_{max}(1+\norm{\mu}_1)+\norm{\lambda}_1$$
In particular, if we fix $\epsilon>0$ and define $\alpha=1/\epsilon^2$, 
$V=\rho/\epsilon$ for some positive constant $\rho$ (such as $\rho=1$), then 
for all $T\geq 1/\epsilon^2$ the right-hand-sides of \eqref{eq:costs}, \eqref{eq:Dav}, \eqref{eq:Dav2} are $O(\epsilon)$, so that  \eqref{eq:perf1}-\eqref{eq:perf4} hold. 
\end{thm}

\begin{proof}  Fix $t \in \{0, 1, 2, \ldots\}$. Rearranging \eqref{eq:result-prev-lem} gives 
\begin{align}
\expect{\Delta(t)}  + V\expect{\pi(t-1)^{\top}G_0(t-1)}&\leq b + Vc_0^* + \alpha\expect{D(\pi^*;\pi(t-1))-D(\pi^*;\pi(t))} \nonumber\\
&- \alpha \expect{D(\pi(t);\pi(t-1))}\nonumber\\
& + V\expect{(\pi(t-1)-\pi(t))^{\top}G_0(t-1)} \label{eq:lemma6sub0} 
\end{align}
To bound terms on the right-hand-side, we have 
\begin{align}
&-\alpha D(\pi(t);\pi(t-1)) \leq -\frac{\alpha}{2}\norm{\pi(t)-\pi(t-1)}_1^2 \label{eq:follows-pinsker}\\
&V(\pi(t-1)-\pi(t))^{\top}G_0(t-1) \leq Vc_{max}\norm{\pi(t)-\pi(t-1)}_1\label{eq:follows-other} 
\end{align}
where \eqref{eq:follows-pinsker} holds by the Pinsker inequality \eqref{eq:pinsker}; \eqref{eq:follows-other} holds because all components of vector $G_0(t-1)$ have magnitude at most $c_{max}$. Substituting these into the right-hand-side of \eqref{eq:lemma6sub0} gives
\begin{align}
\expect{\Delta(t)}  + V\expect{\pi(t-1)^{\top}G_0(t-1)}&\leq b + Vc_0^* + \alpha\expect{D(\pi^*;\pi(t-1))-D(\pi^*;\pi(t))} \nonumber\\
&- \frac{\alpha}{2} \expect{\norm{\pi(t)-\pi(t-1)}_1^2}\nonumber\\
& + Vc_{max}\expect{\norm{\pi(t)-\pi(t-1)}_1}  \label{eq:lemma6sub} 
\end{align}
To bound the last two terms we have 
\begin{align*}
-\frac{\alpha}{2}\norm{\pi(t)-\pi(t-1)}_1^2+ Vc_{max} \norm{\pi(t)-\pi(t-1)}_1&\leq \sup_{x\in\mathbb{R}}\left( -\frac{\alpha}{2}x^2 +Vc_{max}x \right)\\
&= \frac{V^2c_{max}^2}{2\alpha}
\end{align*}
Substituting this into the right-hand-side of \eqref{eq:lemma6sub} yields
\begin{align}
\expect{\Delta(t)}  + V\expect{\pi(t-1)^{\top}G_0(t-1)}&\leq b + Vc_0^* + \alpha\expect{D(\pi^*;\pi(t-1))-D(\pi^*;\pi(t))} + \frac{V^2c_{max}^2}{2\alpha} \label{eq:proof-reuse} 
\end{align}
This holds for all $t \in \{0, 1, 2, \ldots\}$. Fix $T$ as a positive integer. Summing over $t \in \{0, 1, 2, \ldots, T\}$ gives
$$\expect{L(T)-L(0)} + V\sum_{t=-1}^{T-1}\expect{\pi(t)G_0(t)} \leq (b+Vc_0^*)(T+1)  + \alpha \expect{D(\pi^*;\pi(-1))-D(\pi^*;\pi(T))} + \frac{V^2c_{max}^2(T+1)}{2\alpha}$$
Since $L(T)-L(0)\geq 0$, $D(\pi^*;\pi(-1))-D(\pi^*;\pi(T))\leq \log(n)$ (recall \eqref{eq:logn}), and  $\pi(-1)G_0(-1)=-c_{max}$  we have 
$$V\sum_{t=0}^{T-1} \expect{\pi(t)G_0(t)} \leq (b+Vc_0^*)(T+1) + Vc_{max} + \alpha \log(n) + \frac{V^2c_{max}^2(T+1)}{2\alpha}$$
Dividing by $VT$ proves \eqref{eq:costs}.  The remaining inequalities \eqref{eq:queuebound}, \eqref{eq:Dav}, \eqref{eq:Dav2} are proven similarly (see Appendix A).
\end{proof}

\section{Robot simulation} \label{section:simulation} 

\begin{figure}[htbp]
\centering
\begin{tabular}{c | c | c}
 $\alpha$& $\overline{R}_{virtual}$& $\overline{R}_{actual}$\\ \hline \hline
 5 & 0.0039 & 0.1678\\ \hline
 25 & 0.0153 & 0.0066\\ \hline
 50 & 0.6491 & 0.6422 \\ \hline
 100 & 0.6581 & 0.6530 \\ \hline
 1000 &0.6672 &0.6604 
 \end{tabular}
 \caption{Comparing time average cost of the proposed online algorithm over $T=10^6$ slots for the virtual and actual system. The parameter $V=5$ is fixed as $\alpha$ is changed. From the table it seems we need $\alpha\geq |\script{S}|=40$ for cost to be near optimal. For comparison, recall the numerical optimum over the heuristic class of renewal-based policies (fine tuned with knowledge of the problem structure and 
 reward distribution) is $0.66791$.}
 \label{fig:chart} 
\end{figure}

This section simulates the proposed online algorithm for the robot example described in Section \ref{section:robot-example}. For a simulation time $T$ we define the virtual and actual time average reward by 
\begin{align*}
\overline{R}_{virtual} &= \frac{1}{T}\sum_{t=0}^{T-1} \sum_{i\in\script{S}} \pi_i(t)c_{i,0}(W(t), A_i(t))\\
\overline{R}_{actual} &= \frac{1}{T}\sum_{t=0}^{T-1}c_{S(t),0}(W(t),A_{S(t)}(t)) 
\end{align*}
We use the same distribution of rewards as described in Section \ref{section:robot-example-dist}. Using $V=5$, $T=10^6$, the time average rewards for various $\alpha$ values are given in Fig. \ref{fig:chart}. It can be seen that a significant jump in average reward occurs when $\alpha\geq |\script{S}|=40$.  The virtual and actual costs are very close. Also, when $\alpha$ is large, the average cost of both virtual and actual systems seems to be converging close to the value $0.66791$, which is the optimal achieved over the renewal-based Heuristic 2 algorithm of Section \ref{section:robot-example-dist} that is fine tuned with knowledge of the distribution of $W(t)=(W_1(t), \ldots, W_{20}(t))$. We conjecture that Heuristic 2 is in fact optimal for this particular reward distribution.  Our proposed online algorithm yields time averages close to this value without knowing the distribution. 

The time average fractions of time being in each state are shown in Fig. \ref{fig:grid3} for the case $V=5$, $\alpha=1000$. Probabilities are rounded to three places past the decimal point (so probabilities less than $0.0005$ are rounded to $0.000$).    Only results for the virtual system are shown in Fig. \ref{fig:grid3} (the fractions of time for the actual system are similar). The robot learns to avoid location $20$; to avoid being in the home location 1 when it does not have the object (so it almost never stays in the home location more than one slot in a row); to almost never be in the high-value location 9 when it is holding an object (so it immediately transitions out of state 9 once it collects an object there).  It also learns to take 10-hop paths between locations 1 and 9, but it does not take the same path every time. 

\begin{figure}[htbp]
   \centering
   \includegraphics[width=6in]{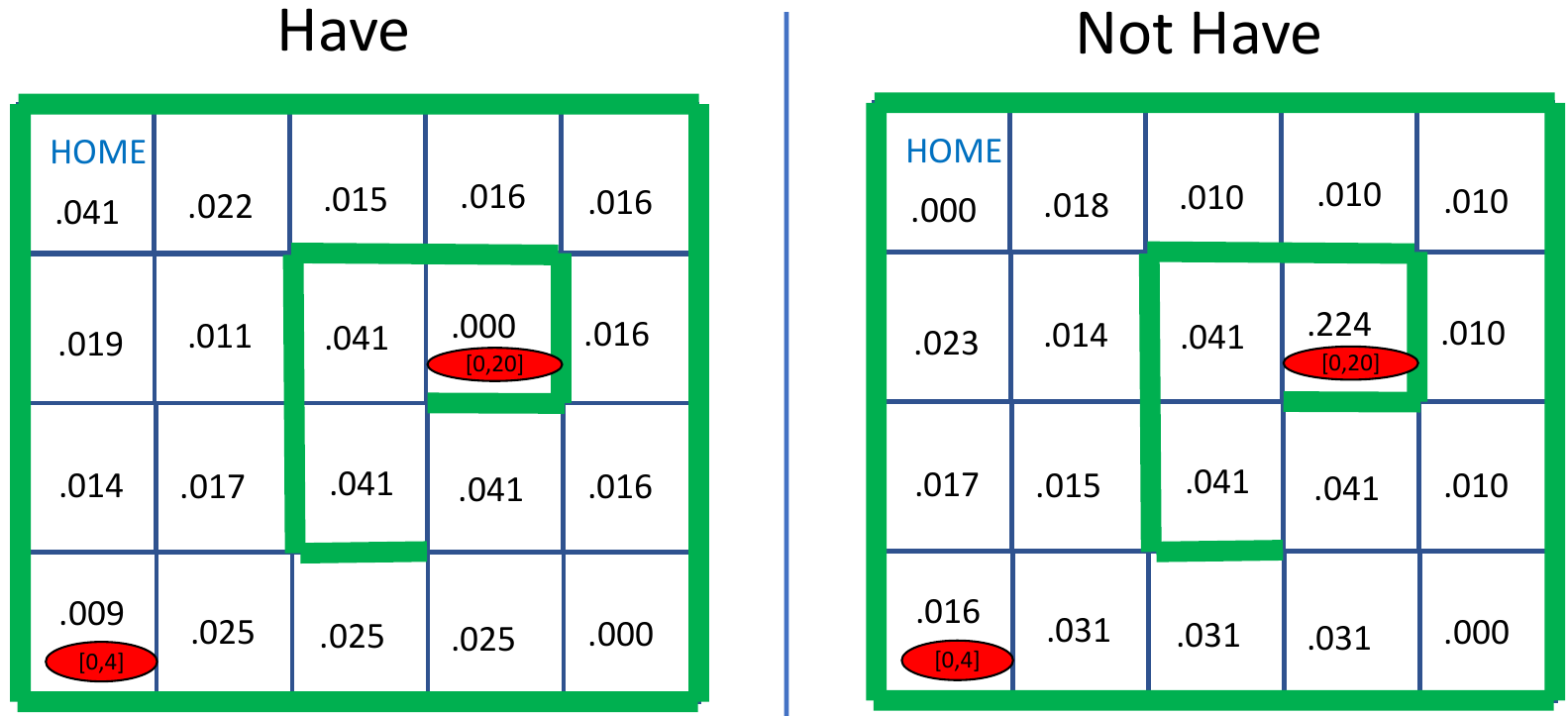} 
   \caption{Fractions of time the virtual system is in each location $a\in\{1, \ldots, 20\}$ when the robot has an object ($Hold(t)=1$) and does not have ($Hold(t)=0$). Parameter settings are $\alpha=1000$, $V=5$, $T=10^6$.}
      \label{fig:grid3}
\end{figure}

\subsection{Dependence on $V$}

Theorem \ref{thm:performance} was stated in terms of general $V>0, \alpha>0$ parameters and is easiest to interpret when $\epsilon>0$ and $\alpha = 1/\epsilon^2$, $V = \rho/\epsilon$ for some constant $\rho>0$. A simple choice is $\rho=1$. However, using a smaller $\rho$ preserves the same theoretical scaling but can provide better performance. This is because, from \eqref{eq:Dav}, deviation in the time averages that relate to the global balance inequalities have some terms that simply go to zero as $T\rightarrow\infty$, while there is a persistent term that is proportional to $V/\alpha$ regardless of the size of $T$. In our simulations it is good to use $\rho \in [1/10, 1/5]$.   Fig. \ref{fig:versusV}(a) fixes $\alpha=1000$ and 
plots time average reward in the virtual and actual systems for $V \in [0, 20]$. The horizontal asymptote is the reward $0.66791$ achieved by the best heuristic in Section \ref{section:robot-example-dist} and fine-tuning the parameters with knowledge of the distribution. It can be seen that there is a close agreement between the virtual and actual system for $V \in [0, 5]$, but the reward of the actual system diminishes as the ratio $V/\alpha$ is increased, as predicted by Theorem \ref{thm:performance}.  Alternatively, the deviation seen in Fig. \ref{fig:versusV}(a) can be fixed by increasing both $\alpha$ and $T$ by a factor of 4, as shown 
in Fig. \ref{fig:versusV}(b).

\begin{figure}%
    \centering
    \subfloat[$\alpha=1000, \: T=10^6$]{{\includegraphics[width=3.3in]{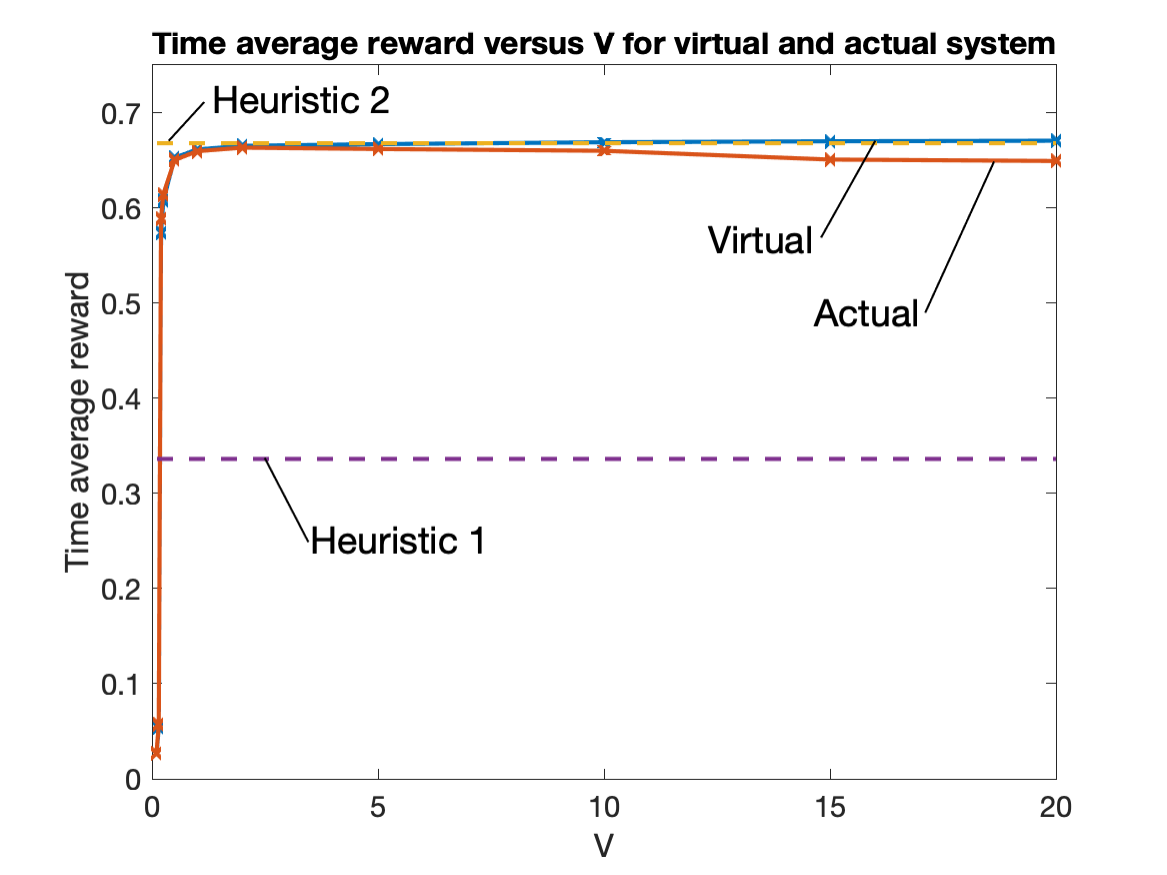} }}%
    \qquad
    \subfloat[$\alpha=4000,\: T=4\times10^6$]{{\includegraphics[width=3.3in]{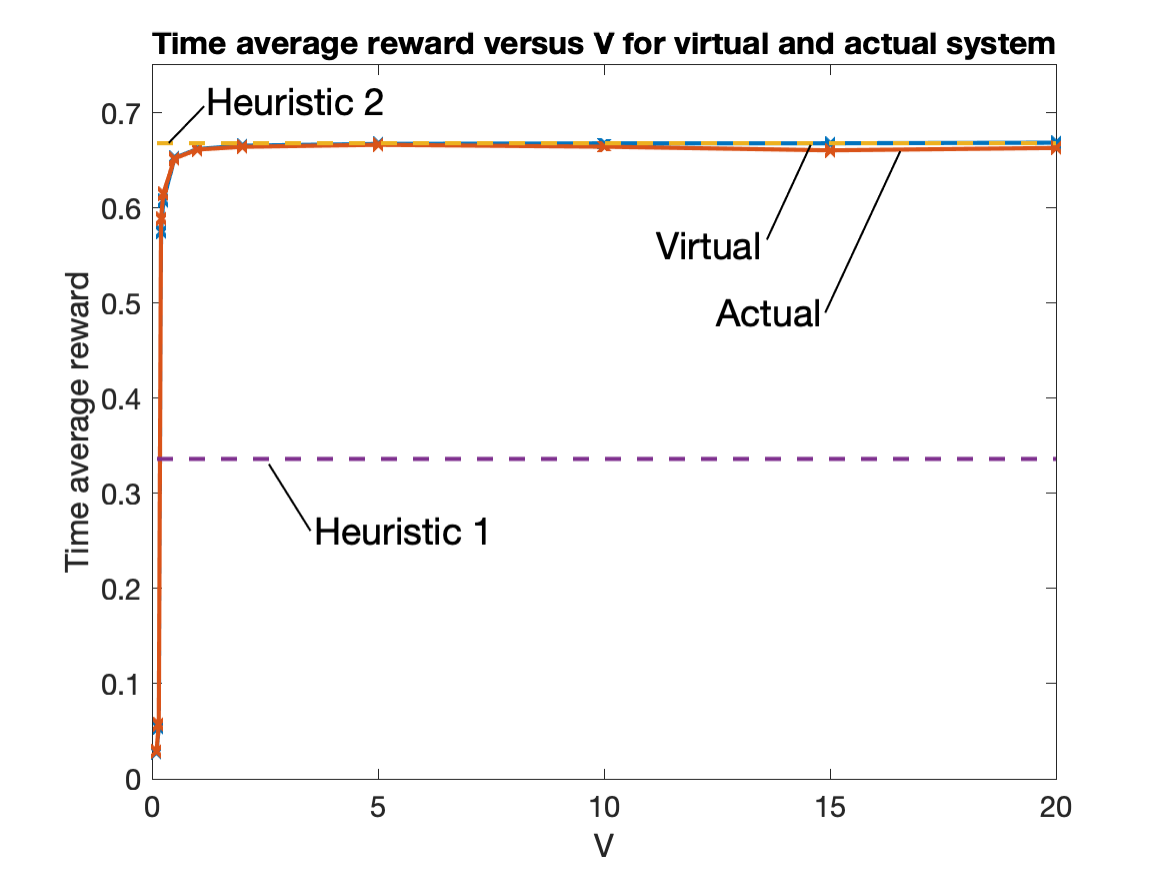} }}%
    \caption{(a) Time average cost versus $V$ for virtual and actual systems ($\alpha=1000$, $T=10^6$). The virtual and actual systems have similar performance when the ratio $V/\alpha$ is small, as predicted by Theorem \ref{thm:performance}; (b) Removing deviation between virtual and actual systems by increasing $\alpha$ and $T$ by a factor of 4 ($\alpha=4000$, $T=4 \times 10^6$).}%
    \label{fig:versusV}%
\end{figure}

\subsection{Different distributions}

\begin{figure}%
    \centering
    \subfloat[Without Redirect mode]{{\includegraphics[width=3.3in]{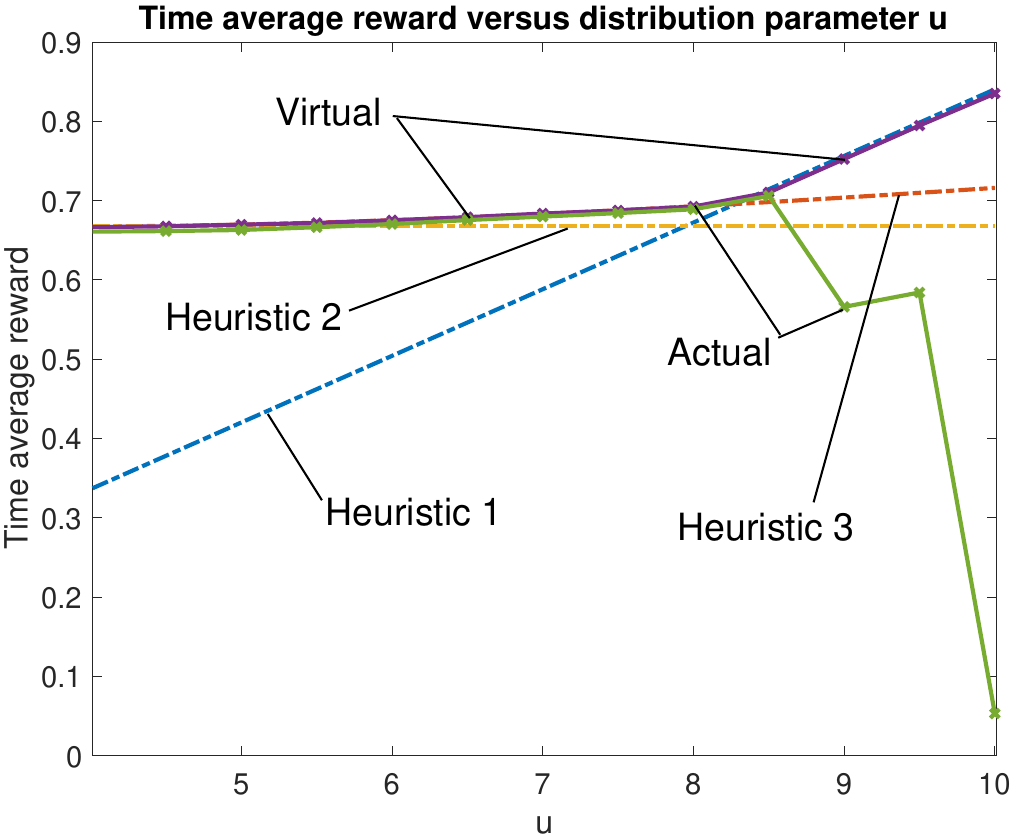} }}%
    \qquad
    \subfloat[With Redirect mode]{{\includegraphics[width=3.3in]{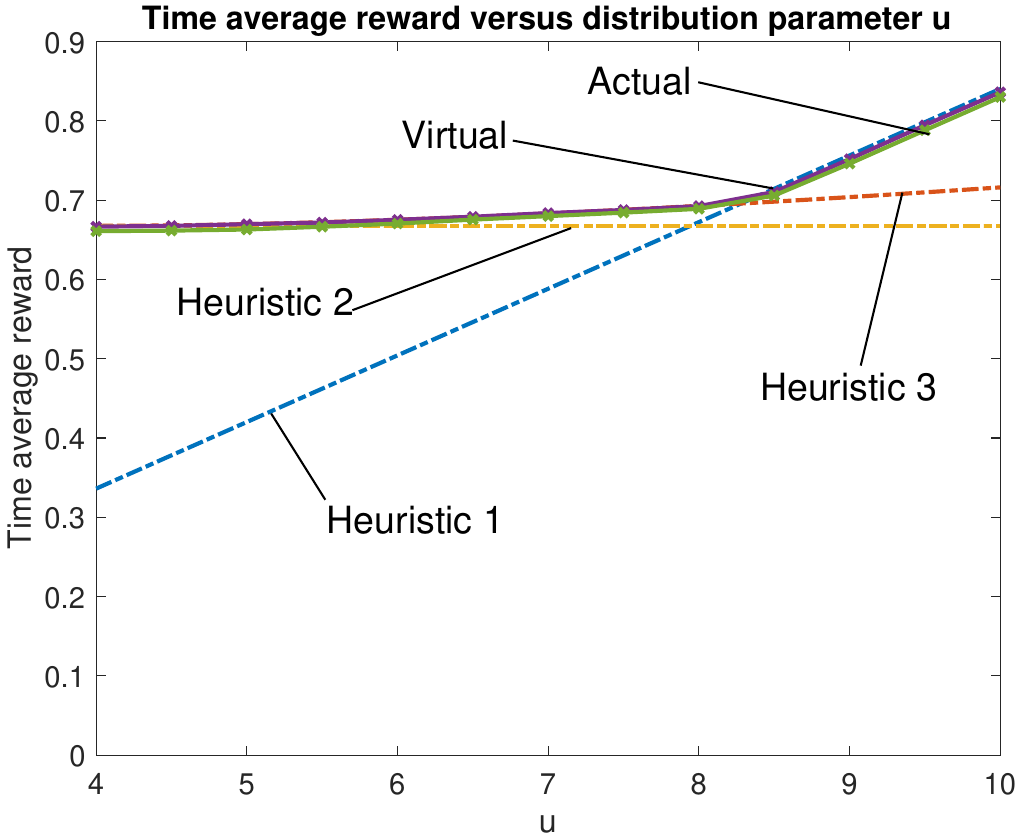} }}%
    \caption{(a) Time average reward versus distribution parameter $u\in[4, 10]$ for $V=5, \alpha=1000, T=10^6$. The three dashed near-linear curves plot $\overline{r}(u)$ for Heuristics 1, 2, 3. The purple curve is the proposed online virtual algorithm and achieves the max of the three heuristics for each $u \in [4, 10]$. The green curve is the actual online algorithm and matches the virtual performance for $u \in [4, 8.5]$ but falls into trapping states for $u\in\{9, 9.5, 10\}$; (b) The same setup, with the exception that the actual system uses the Redirect mode when it detects a trapping state. Performance for both virtual and actual systems closely match and lie on the max of the three heuristic $\overline{r}(u)$ curves.}%
    \label{fig:redirect}%
\end{figure}

Now change the rewards in location $16$ so $R_{16}(t) \sim \mbox{Unif}[0, u]$ for some parameter $u \in [4, 10]$ (the previous subsection uses special case $u=4$). We compare the proposed online algorithm with the following three distribution-aware heuristics. 

\begin{itemize} 
\item Heuristic 1 (with parameter $u$): This is the same as described in Section \ref{section:robot-example-dist}
with the exception that $\theta \in [0, u)$.  Then 
$$ \overline{r}(u) = \sup_{\theta \in [0, u)} \left\{ \frac{\frac{1}{2}(\theta + u)}{5 + \frac{2u}{u-\theta}}\right\}$$
where $\overline{r}(u)$ is the best time average reward for this class of strategies, considering all parameters $\theta \in [0, u)$. 
\item Heuristic 2: Same as before, with $\overline{r}(u)=0.66791$ for all $u$.
\item Heuristic 3: Starting from location 1, move  to location 16 in 3 steps via the path $1\rightarrow6\rightarrow 11\rightarrow 16$ (ignoring any rewards that appear in locations $6$ and $11$). Stay in location 16 for only one slot: If during that one slot we see $W_{16}(t)>\theta_1$ (for some parameter $\theta_1 \in [0, u]$) then collect this object and go back to location 1 in three more steps (ending the frame in 6 slots). Else, take the 7 hop path from 16 to 9 (ignoring rewards along the way). Stay in location 9 until we see a reward $W_9(t)>\theta_2$ (for some $\theta_2 \in [0, 20)$). Collect this reward and return home using any 10-hop path. Repeat. On each frame, the probability of collecting an object from location  $16$ (and hence having a 6-slot frame) is $\frac{1}{2}\left(\frac{u-\theta_1}{u}\right)$ and so by renewal-reward theory: 
$$\overline{r}(u)=\sup_{\theta_1\in[0,u], \theta_2 \in [0, 20)}\left\{\frac{\frac{1}{2}\left(\frac{u-\theta_1}{u}\right)\left(\frac{\theta_1+u}{2}\right) + \left(1-\frac{1}{2}\left(\frac{u-\theta_1}{u}\right)\right)\frac{\theta_2+20}{2}}{\frac{1}{2}\left(\frac{u-\theta_1}{u}\right)6 + \left(1 - \frac{1}{2}\left(\frac{u-\theta_1}{u}\right)\right)(19 + \frac{40}{20-\theta_2})  } \right\}$$
\end{itemize}

Fig. \ref{fig:redirect}(a) plots $\overline{r}(u)$ versus $u \in [4, 10]$ for these three distribution-aware heuristics. The curves have a near-linear structure in the figure. Fig. \ref{fig:redirect}(a) also plots data for the proposed online algorithm, which does not know the distribution, using parameters $\alpha=1000$, $V=5$, $T=10^6$.  Each data point of the virtual system lies on the maximum of the three heuristic curves.  The time average reward in the actual system matches the virtual system for $u \in [4, 8.5]$, but significantly deviates for larger values of $u$ due to falling into a \emph{trapping state}. This is easily fixed by applying the Redirect feature described below (see Fig. \ref{fig:redirect}(b)). 

The trapping issue can be understood by observing the fractions of time in each state when $u=8$ (where virtual and actual data points in Fig. \ref{fig:redirect}(a) are similar) and $u=9$ (where virtual and actual data points  in Fig. \ref{fig:redirect}(a) are different). 

\begin{enumerate} 
\item No traps emerge: Fig. \ref{fig:gridWithoutRedirect} shows the fractions of time of the virtual system when $u=8$ (the fractions of time in the actual system are similar and are not shown). From the ``Not have'' data in Fig. \ref{fig:gridWithoutRedirect}, when it does not have an object, the robot almost never visits locations outside of the path 
$$1\rightarrow 6\rightarrow 11\rightarrow 16\rightarrow 17\rightarrow 18\rightarrow 19\rightarrow 14\rightarrow 13\rightarrow 8\rightarrow 9$$  
In fact, it learns to take actions similar to Heuristic 3: It almost always uses the same 3-hop path from $1$ to $16$. It then waits in location 16 for a small amount of time (slightly more than the 1-slot wait of Heuristic 3). If it sees a sufficiently large valued object in location 16 before it finishes its wait, it picks it up and goes home via the 3-hop path $16\rightarrow 11\rightarrow 6\rightarrow 1$. Else, it traverses the path $16\rightarrow17\rightarrow18\rightarrow  19\rightarrow14\rightarrow13\rightarrow8\rightarrow 9$.   It waits in location 9 until it sees an object that it views as being sufficiently valuable. It then collects this object and returns home on any 10-hop path. It does not always use the same path home, which is why Fig. \ref{fig:gridWithoutRedirect} shows nonzero probabilities for many locations under the ``Have'' states. The robot also learns efficient thresholds for deciding when an object is valuable enough to collect.  

\item Traps emerge: Figs.  \ref{fig:gridu9virtual} and \ref{fig:gridu9actual} consider the case $u=9$ and compare the fractions of time the virtual and actual system are in each basic state. There is a significant difference here, which explains the deviant behavior for $u=9$ in the curves of  Fig. \ref{fig:redirect}(a).  Consider Fig. \ref{fig:gridu9virtual}: Rewards in location 16 are now so high that the virtual system learns to avoid locations outside the set $\{1, 6, 11, 16\}$. 
It learns 
to use the strategy of starting from home, taking the path $1\rightarrow 6 \rightarrow 11\rightarrow 16$, waiting in location 16 until it sees an object of sufficiently large value, then returning home via $16\rightarrow 11\rightarrow 6 \rightarrow 1$. This creates a steady state distribution with many ``isolated'' states, being states that have near-zero probability and whose neighbors have near-zero probability.  In contrast, the fractions of time for the actual system are shown in Fig. \ref{fig:gridu9actual}. It is clear that the robot gets trapped in location 20, where it consistently chooses to stay.  The virtual system has (correctly) learned to avoid location 20, so its contingency actions for location 20 did not train enough to become efficient (they repeatedly are the inefficient action to remain in location 20). The actual robot rarely enters location 20, but when it does it uses the inefficient contingency actions of staying there. So it becomes trapped and earns near-zero reward.  This ``always stay in location 20'' strategy yields  a valid but undesirable solution to  the global balance equations. Getting trapped is a nonergodic event on the actual system, and so the overall average reward in the actual system changes from simulation to simulation depending on when the robot gets trapped.   Fortunately, such trap situations  are easy to detect: The robot observes that it spends an excessive amount of time in a state that the virtual system has assigned near-zero probability! 
The Redirect mode fixes the problem (as shown in Fig. \ref{fig:redirect}(b)).  
\end{enumerate}

\begin{figure}[htbp]
   \centering
   \includegraphics[width=6in]{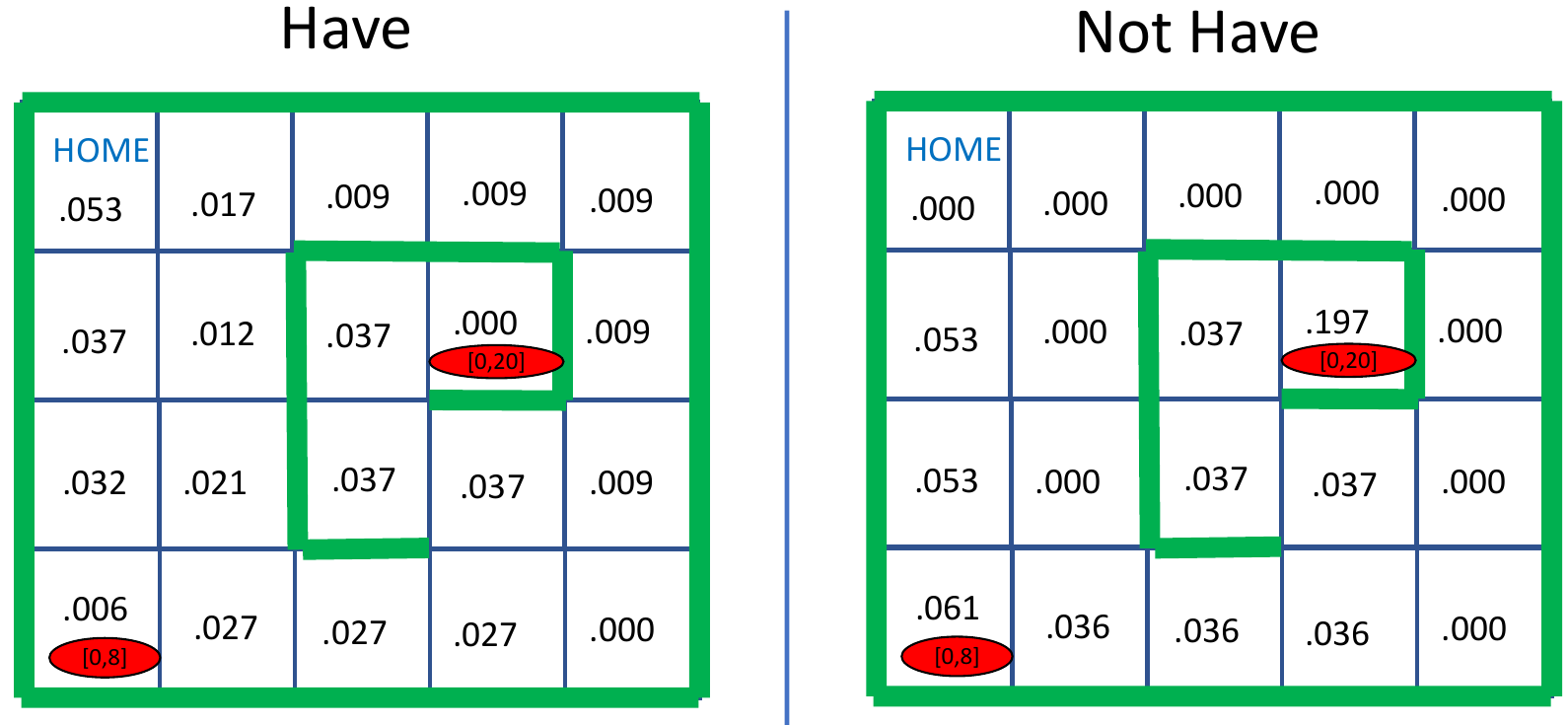} 
   \caption{Virtual system ($u=8$): Fractions of time the virtual system is in each basic state for the $u=8$ simulation of Fig. \ref{fig:redirect}(a).}
   \label{fig:gridWithoutRedirect}
\end{figure}

\begin{figure}[htbp]
   \centering
   \includegraphics[width=5in]{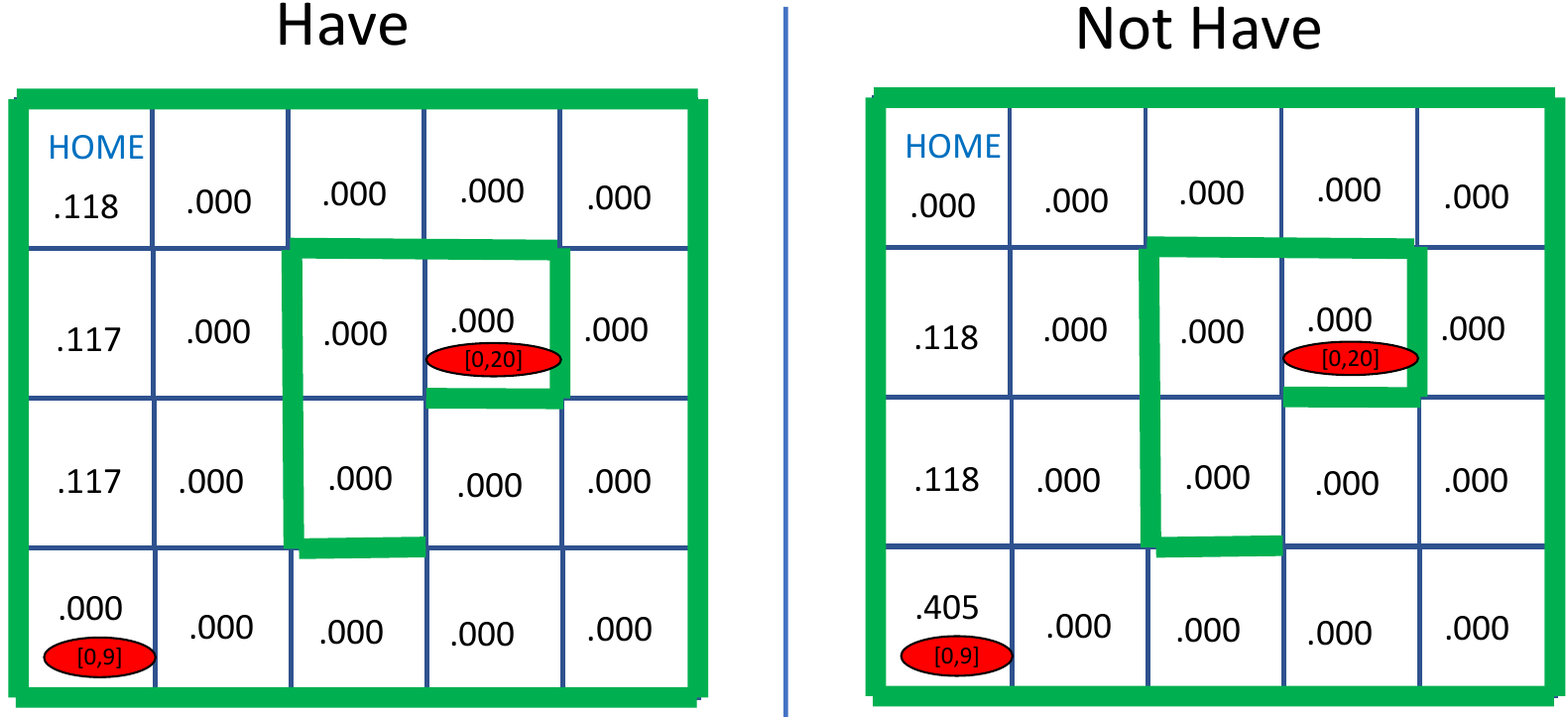} 
   \caption{Virtual system ($u=9$): Fractions of time the virtual system is in each basic state for the $u=9$ simulation of Fig. \ref{fig:redirect}(a). The virtual system correctly learns that it is not worthwhile to pursue the high-valued rewards in location 9. Rather, it learns to almost always be in locations in the set $\{1, 6, 11, 16\}$, which are locations on the West edge of the region.}
   \label{fig:gridu9virtual}
\end{figure}

\begin{figure}[htbp]
   \centering
   \includegraphics[width=5in]{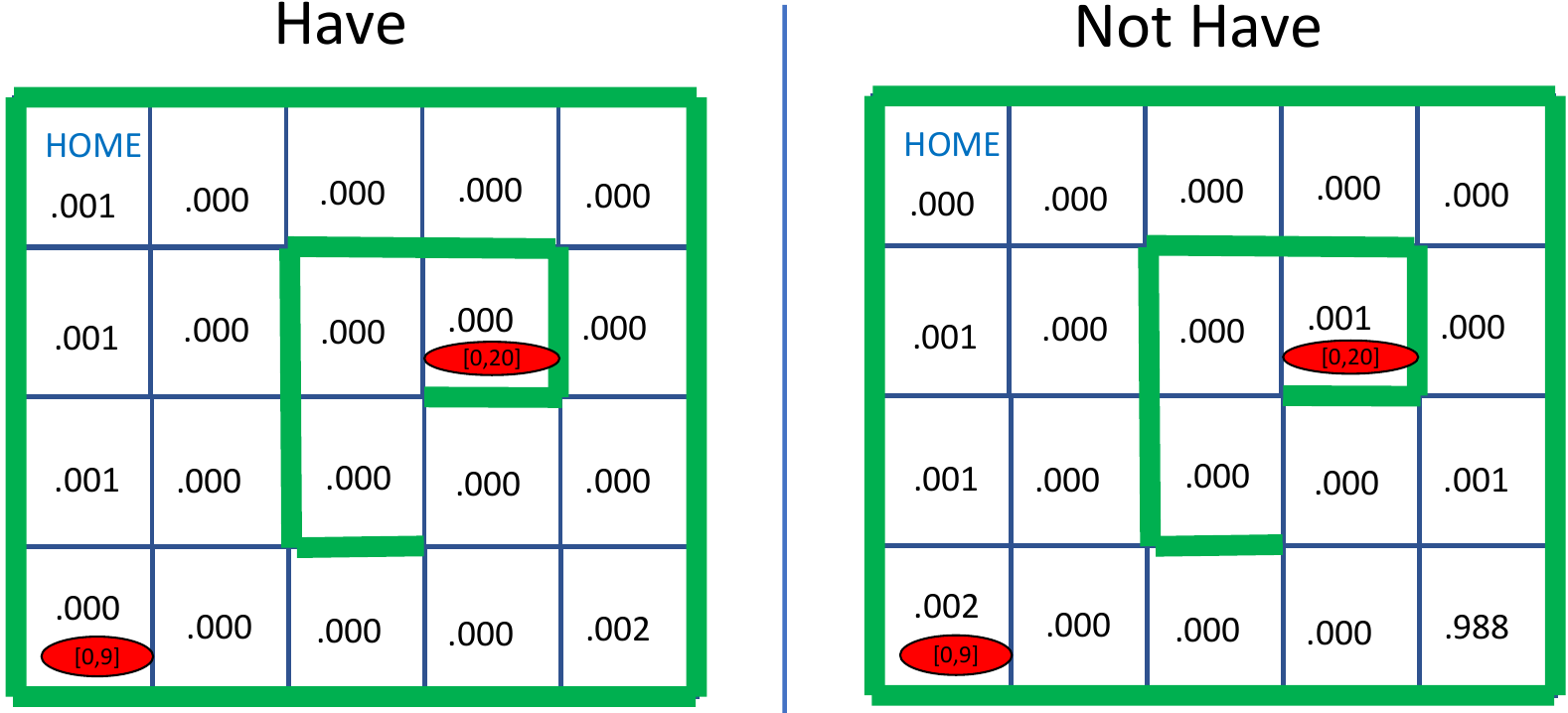} 
   \caption{Actual system for case $u=9$: Fractions of time the actual system is in each basic state for the $u=9$ simulation of Fig. \ref{fig:redirect}(a). The robot gets trapped in the undesirable state $(20,0)$, where its contingency action $A_{20}(t)$ tells it to stay in location $20$. The virtual contingency actions for state $(20,0)$ never learn to be efficient because the virtual system (correctly) allocates near-zero probability to location 20. This explains the low reward shown in Fig. \ref{fig:redirect}(a) for the actual system for the data points $u \in \{9, 9.5, 10\}$. The issue is fixed, leading to similar behavior in the virtual and actual systems, by introducing the Redirect mode (see Fig. \ref{fig:redirect}(b)).}
   \label{fig:gridu9actual}
\end{figure}

The Redirect mode is implemented in the actual system as follows: The robot maintains a running average time in each state over some window of past slots. This could be done, for example, using a window of the past 1000 slots, or using an exponentiated average with $\gamma=1/1000$ (we choose the latter in the simulation). If the actual robot is currently in a state with actual exponentiated average greater than some threshold $\theta_{high}$ (we use $\theta_{high}=1/10$ in this simulation), and the corresponding virtual exponentiated average is less than some threshold $\theta_{low}$ (we use $\theta_{low}=0.00001$ in this simulation), then it enters ``Redirect Mode'': In this mode, it ignores the contingency actions the virtual system tells it to take. Instead, it takes a direct path back to location 1, where it then exits redirect mode and follows the actions given to it by the virtual system.  In a more general MDP, the redirect mode can be any sequence of actions that ensure the robot travels to a particular desired location, say location 1, in finite expected time. 

\subsection{Average power constraint} \label{section:simulation-power}

Now assume the robot expends 1 unit of power on each slot that it moves without holding an object, 2 units of energy if it moves while holding an object (since the object is heavy), and 0 units when  staying in the same location. The reward distribution is the same as above with $u=4$. We impose an average power constraint 
$$ \lim_{T\rightarrow\infty} \frac{1}{T}\sum_{t=0}^{T-1} power(t)\leq 0.9$$
We use Heuristic 2 for comparison, which by renewal-reward theory yields average reward $\overline{r}$ and average power $\overline{p}$ in terms of $\theta \in [0, 20)$ of
$$\overline{r} = \frac{\frac{1}{2}(\theta+20)}{19 + \frac{40}{20-\theta}} \quad , \quad \overline{p} = \frac{10 + 2(10)}{19 + \frac{40}{20-\theta}}$$
Optimizing $\theta \in [0, 20)$ to maximize $\overline{r}$ subject to $\overline{p}\leq 0.9$ for Heuristic 2 yields $\theta^*=17.2093$ and $\overline{r}^*=0.55814$, $\overline{p}^*=0.9$.  The proposed online algorithm (with Redirect) with $\alpha=1000$, $V=5$, $T=10^6$ yields: 
\begin{itemize} 
\item Virtual system: $(\overline{r}, \overline{p}) = (0.5693, 0.8996)$
\item Actual system $(\overline{r}, \overline{p}) = (0.5559, 0.8859)$. 
\end{itemize} 
which is competitive with the heuristic, suggesting the heuristic is near-optimal.  Across independent simulations, the average reward of the virtual system is similar to the reported value $0.5693$ above, which is slightly larger than the  $0.55814$ value of the heuristic, suggesting that true optimality is slightly better than the heuristic.  We expected more significant digits to match between the virtual and actual systems, a different set of trap detection thresholds may improve this discrepancy.

\subsection{Comparison with stochastic approximation on a value function} 

This section compares to stochastic approximation on a cost-to-go function, also called a value function. The value function does not incorporate additional time average constraints $\overline{C}_l\leq 0$, so we compare on the robot example without the average power constraint (that is, $k=0$).  For simplicity of the online comparison, our value function method shall approximate the time average reward problem by a discounted problem with discount factor $\rho = 999/1000$.  A traditional value function method 
would use the full state vector $(S(t),W(t))$, which has large dimension with infinite possibilities, see for example \cite{bertsekas-neural}\cite{bertsekas-dp}\cite{puterman}. Temporal difference techniques could be used to approximate a value function on  $(S(t),W(t))$ according to a function with simple structure \cite{bertsekas-neural}. However, 
a fair comparison directly uses the opportunistic MDP structure of our problem that allows a value function $J(i)$ to be defined only on states $i \in \script{S}$, where $J(i)$ is the optimal expected discounted reward given we start in state $i$. Then
\begin{align}\label{eq:DP}
J(i) = \expect{\max_{A_i \in \script{A}_i} \left(-c_{i,0}(W(t), A_i(t))  + \rho \sum_{j\in\script{S}} p_{i,j}(W(t),A_i(t))J(j) \right)} \quad \forall i \in \script{S}
\end{align} 
where the inner maximization is done with knowledge of $W(t)$, and the expectation is taken with respect to the distribution of $W(t)$.  Since the distribution of $W(t)$ is unknown we use the following online stochastic approximation: Initialize $J_{-1}(i)=0$ for $i \in \script{S}$.  On each step $t \in \{0, 1, 2, \ldots\}$ do
\begin{itemize} 
\item Observe $W(t)$. For each $i \in \script{S}$ compute $A_i(t) \in \script{A}_i$ as the maximizer of 
$$-c_{i,0}(W(t), A_i(t)) + \rho \sum_{j \in \script{S}}p_{i,j}(W(t),A_i(t))J_{t-1}(j)$$
treating $W(t)$ and $J_{t-1}(j)$ as known constants. 
Define $\tilde{J}_t(i)$ as the maximized value of the above expression. 
\item For $i \in \script{S}$ update the value function by 
\begin{equation} \label{eq:RM}
 J_t(i) = (1-\eta) J_{t-1}(i) + \eta \tilde{J}_t(i) \quad \forall i \in \script{S}
 \end{equation} 
where $\eta \in (0,1)$ is a stepsize. 
\item Given $S(t)=i$, apply action $A_i(t)$ as computed above. 
\end{itemize} 
The complexity of this online value function based algorithm is competitive with our proposed approach. In particular, this method maintains a value $J_t(i)$ for each $i \in \script{S}$, while the proposed method (with no additional cost constraints) uses a virtual queue $Q_i(t)$ for each $i \in \script{S}$.  This value function based approach is presented here as a heuristic: At best it approximates a solution to \eqref{eq:DP}, which is itself a discounted approximation to the infinite horizon time average problem of interest. Nevertheless,  the iteration \eqref{eq:RM} resembles a classic Robbins-Monro stochastic approximation (see \cite{robbins-monro}) and can likely be analyzed according to such techniques. Further, this value function based method (called the $J$-based method in Fig. \ref{fig:J-compare}) simulates remarkably well for the robot problem with no additional time average constraints (that is, no average power constraint). Indeed, using $\gamma=1/1000$ and duplicating the scenario of Fig. \ref{fig:redirect}b, we simulate this $J$-based method and compare to the actual and virtual rewards of the proposed algorithm (where actual rewards use the redirect mode). The results are shown in Fig. \ref{fig:J-compare}, which shows three curves that look very similar, all appearing to reach near optimality (where optimality is defined by Theorems \ref{thm:converse} and \ref{thm:achievability}).  Of course, this $J$-based heuristic cannot handle extended problems with time average inequality constraints, while our proposed algorithm handles these easily. 

\begin{figure}[htbp]
   \centering
   \includegraphics[height=3in]{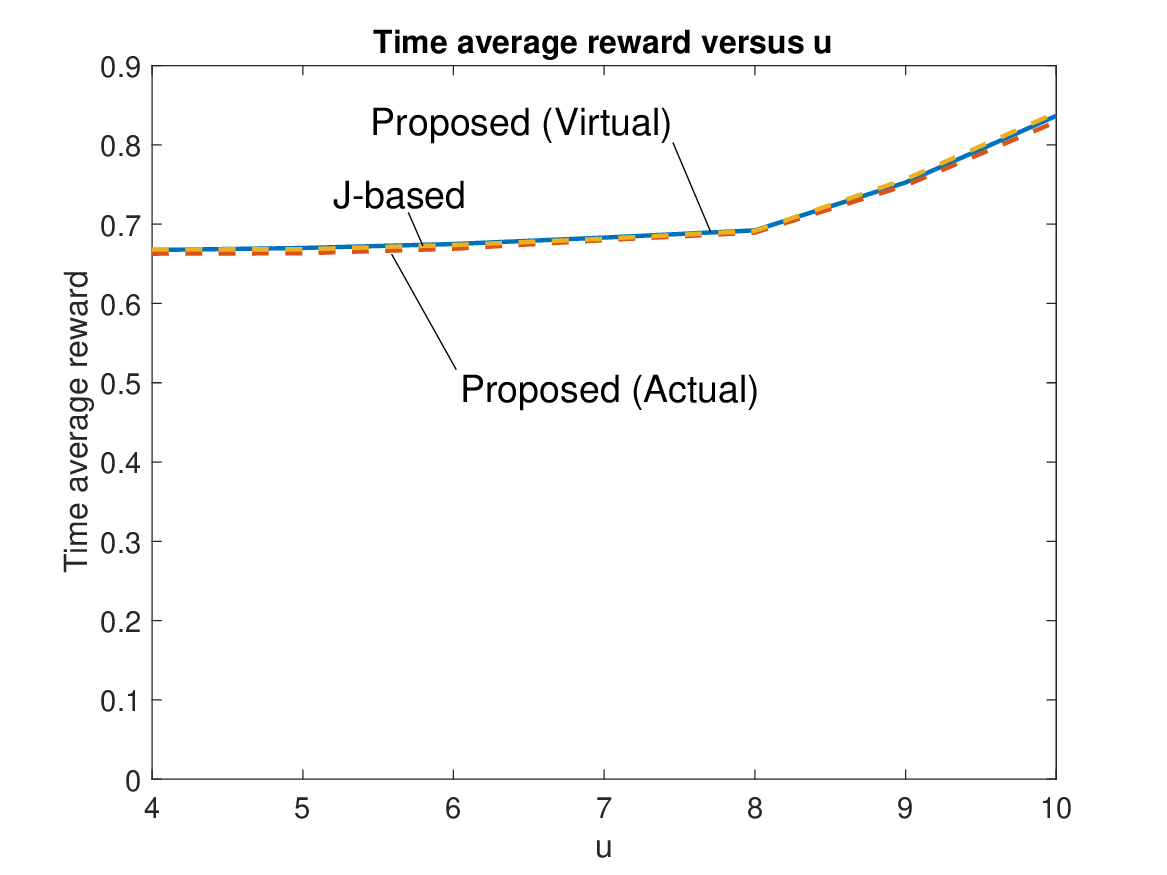} 
   \caption{Comparing the proposed algorithm with the $J$-based heuristic for the same scenario as Fig. \ref{fig:redirect}b (parameter $u \in [4, 10]$). Proposed virtual  is solid blue; proposed actual is dashed red; $J$-based is dashed yellow.}
   \label{fig:J-compare}
\end{figure}

\section{Conclusion} 

This work extends the max-weight and drift-plus-penalty methods of stochastic network optimization to opportunistic
Markov decision systems. The basic state variable $S(t)$ can take one of $n$ values, but the full
state $(S(t),W(t))$ is augmented by a sequence of i.i.d. random vectors $\{W(t)\}$ that can be observed
at the start of each slot $t$ but have an unknown distribution. Learning in this system is mapped to the problem of achieving a collection of time average global balance constraints.  The resulting algorithm operates on a virtual and actual system at the same time. For any $\epsilon>0$, parameters of the algorithm can be chosen to ensure performance in the virtual system is within $O(\epsilon)$ of optimality after a convergence time of $1/\epsilon^2$.  The coefficient multiplier, and hence overall system performance,  depends only on the number of basic states $n$ and is independent of the dimension of $W(t)$ and the (possibly infinite) number of values $W(t)$ can take. The actual system is shown mathematically to have the same conditional probabilities and expectations as the virtual system, and is shown in simulation to closely match the virtual system in its time average rewards and fractions of time in each basic state.  The online algorithm is augmented with a Redirect mode to detect and alleviate issues regarding trapping states that relate to non-irreducible situations.  Simulations on the robot example show  the proposed online algorithm, which does not know the distribution of the reward vector, has time averages that are close to 
the optimum over a class of heuristic renewal-based algorithms that are fine tuned with knowledge of the problem structure and the probability distribution.

\section*{Appendix A -- Proof of final part of Theorem \ref{thm:performance}}

\begin{proof} (Inequality \eqref{eq:queuebound}) Rearranging \eqref{eq:lemma6sub} gives 
\begin{align}
\expect{\Delta(t)} &\leq b + Vc_0^*-V\expect{\pi(t-1)^{\top}G_0(t-1)} \nonumber\\
&\quad +  \alpha\expect{D(\pi^*;\pi(t-1))-D(\pi^*;\pi(t))} \\
&\quad - \frac{\alpha}{2}\expect{\norm{\pi(t)-\pi(t-1)}_1^2} + Vc_{max}\expect{\norm{\pi(t)-\pi(t-1)}_1} \label{eq:sub56} 
\end{align}
Since $\pi(t-1)$ is independent of $W(t-1)$, the Lagrange multiplier assumption \eqref{eq:LM-consequence} implies 
\begin{align}
c_0^* - \expect{\pi(t-1)^{\top}G_0(t-1)}  &\leq \expect{\pi(t-1)Y(t-1)}\lambda+\expect{\pi(t-1) G(t-1)}\mu \label{eq:L1}\\
&\leq\expect{\pi(t)Y(t-1)}\lambda+\expect{\pi(t) G(t-1)}\mu \nonumber \\
&\quad + (\norm{\lambda}_1 + c_{max}\norm{\mu}_1)\norm{\pi(t)-\pi(t-1)}_1  \label{eq:L2} \\
&\leq \expect{Q(t+1)-Q(t)}\lambda+\expect{Z(t+1)-Z(t)}\mu \nonumber \\
&\quad + (\norm{\lambda}_1 + c_{max}\norm{\mu}_1)\norm{\pi(t)-\pi(t-1)}_1 \label{eq:L3}
\end{align}
where \eqref{eq:L1} holds by writing the Lagrange multiplier inequality \eqref{eq:LM-consequence}  in a simpler matrix form; \eqref{eq:L2} holds because all components of matrix $Y(t-1)$ have magnitude at most $1$, all components of matrix $G(t-1)$ have magnitude at most $c_{max}$; \eqref{eq:L3} holds by the update equations \eqref{eq:q-update}, \eqref{eq:z-update} and the fact $\mu_l\geq 0$ for $l\in\{1, \ldots, k\}$.

Substituting \eqref{eq:L3} into \eqref{eq:sub56} gives 
\begin{align}
\expect{\Delta(t)} &\leq b + V\expect{Q(t+1)-Q(t)}\lambda+V\expect{Z(t+1)-Z(t)}\mu \nonumber\\
&\quad +  \alpha\expect{D(\pi^*;\pi(t-1))-D(\pi^*;\pi(t))} \nonumber \\
&\quad - \frac{\alpha}{2}\expect{\norm{\pi(t)-\pi(t-1)}_1^2} + Vd\expect{\norm{\pi(t)-\pi(t-1)}_1} \label{eq:Thmb}
\end{align}
where $d=c_{max}(1+\norm{\mu}_1) + \norm{\lambda}_1$.  The last two terms on the right-hand-side have the bound
$$-\frac{\alpha}{2}\norm{\pi(t)-\pi(t-1)}_1^2 + Vd\norm{\pi(t)-\pi(t-1)}_1 \leq \sup_{x\in\mathbb{R}}\left(-\frac{\alpha}{2}x^2 + Vdx \right) = \frac{V^2d^2}{2\alpha} $$
Substituting this into the right-hand-side of \eqref{eq:Thmb} gives
\begin{align}
\expect{\Delta(t)} &\leq b + V\expect{Q(t+1)-Q(t)}\lambda+V\expect{Z(t+1)-Z(t)}\mu \nonumber\\
&\quad +  \alpha\expect{D(\pi^*;\pi(t-1))-D(\pi^*;\pi(t))}+ \frac{V^2d^2}{2\alpha} \label{eq:Thmb2}
\end{align}
Summing \eqref{eq:Thmb2} over $t \in \{0, \ldots, T-1\}$ and using $L(0)=0$, $Q(0)=0$, $Z(0)=0$ and $D(\pi^*; \pi(-1))-D(\pi^*, \pi(t))\leq \log(n)$ gives
$$\expect{L(T)} \leq bT + V\expect{Q(T)}\lambda + V\expect{Z(T)}\mu + \alpha \log(n) + \frac{V^2d^2T}{2\alpha}$$
Substituting the definition $L(T)=\frac{1}{2}\norm{J(T)}^2$, where $J(T)=(Q(T);Z(T))$, and using the Cauchy-Schwarz inequality yields 
$$\frac{1}{2}\expect{\norm{J(T)}^2} \leq bT + V\expect{\norm{J(T)}}\cdot\norm{(\lambda;\mu)} + \alpha \log(n) + \frac{V^2d^2T}{2\alpha}$$
Define $y=\expect{\norm{J(T)}}$. By Jensen's inequality we have $y^2\leq \expect{\norm{J(T)}^2}$ and so 
\begin{equation} \label{eq:quadratic} 
\frac{1}{2}y^2 \leq Vy\norm{(\lambda;\mu)} + r
\end{equation} 
where we define 
$$ r = bT + \alpha\log(n) + \frac{V^2d^2T}{2\alpha}$$
The quadratic formula applied to \eqref{eq:quadratic}  implies 
$$ y \leq V\norm{(\lambda;\mu)} + \sqrt{V^2\norm{(\lambda;\mu)}^2 +2r}$$
which proves \eqref{eq:queuebound}.
\end{proof} 

\begin{proof} (Inequalities \eqref{eq:Dav} and \eqref{eq:Dav2}) Rearranging \eqref{eq:Thmb} gives 
\begin{align*}
\frac{\alpha}{4}\expect{\norm{\pi(t)-\pi(t-1)}_1^2} &\leq b - \expect{\Delta(t)} + V\expect{Q(t+1)-Q(t)}\lambda + V\expect{Z(t+1)-Z(t)}\mu\\
&\quad + \alpha\expect{D(\pi^*;\pi(t-1))-D(\pi^*;\pi(t))}\\
&\quad - \frac{\alpha}{4}\expect{\norm{\pi(t)-\pi(t-1)}_1^2} + Vd\expect{\norm{\pi(t)-\pi(t-1)}_1}
\end{align*}
To bound the final terms on the right-hand-side, we have 
$$-\frac{\alpha}{4}\norm{\pi(t)-\pi(t-1)}_1^2 + Vd\norm{\pi(t)-\pi(t-1)}_1 \leq \sup_{x\in\mathbb{R}}\left(-\frac{\alpha}{4}x^2+ Vdx\right) = \frac{V^2d^2}{\alpha}$$
Thus
\begin{align*}
\frac{\alpha}{4}\expect{\norm{\pi(t)-\pi(t-1)}_1^2} &\leq b - \expect{\Delta(t)} + V\expect{Q(t+1)-Q(t)}\lambda + V\expect{Z(t+1)-Z(t)}\mu\\
&\quad + \alpha\expect{D(\pi^*;\pi(t-1))-D(\pi^*;\pi(t))} + \frac{V^2d^2}{\alpha}
\end{align*}
Summing over $t \in \{0, \ldots,T\}$ and using $J(0)=0$ gives
\begin{align*}
\frac{\alpha}{4}\sum_{t=0}^T\expect{\norm{\pi(t)-\pi(t-1)}_1^2} \leq (b + \frac{V^2d^2}{\alpha})(T+1)- \frac{1}{2}\expect{\norm{J(T+1)}^2}+ V\expect{J(T+1)^{\top}(\lambda;\mu)}  +\alpha \log(n)
\end{align*}
To bound terms on the right-hand-side, we have 
\begin{align*}
-\frac{1}{2}\norm{J(T+1)}^2 + VJ(t+1)^{\top}(\lambda;\mu) &\leq \sup_{x \in\mathbb{R}}\left(-\frac{x^2}{2} + Vx\norm{(\lambda;\mu)} \right) = \frac{V^2\norm{(\lambda;\mu)}^2}{2}
\end{align*}
Substituting this into the previous inequality gives
\begin{align}
\frac{\alpha}{4}\sum_{t=0}^T\expect{\norm{\pi(t)-\pi(t-1)}_1^2} \leq (b + \frac{V^2d^2}{\alpha})(T+1)+\alpha \log(n) + \frac{V^2\norm{(\lambda;\mu)}^2}{2} \label{eq:independent-interest}
\end{align}
Neglecting the $t=0$ term on the left-hand-side, dividing both sides by $T$, and using Jensen's inequality gives 
\begin{align*}
\frac{\alpha}{4}\left(\frac{1}{T}\sum_{t=1}^T\expect{\norm{\pi(t)-\pi(t-1)}_1}\right)^2 \leq (b + \frac{V^2d^2}{\alpha})(1+\frac{1}{T})+\frac{\alpha \log(n)}{T} + \frac{V^2\norm{(\lambda;\mu)}^2}{2T}
\end{align*}
and so 
$$ \frac{1}{T}\sum_{t=1}^T\expect{\norm{\pi(t)-\pi(t-1)}_1} \leq \sqrt{(\frac{4b}{\alpha}+ \frac{4V^2d^2}{\alpha^2})(1+\frac{1}{T}) + \frac{4\log(n)}{T} + \frac{2V^2\norm{(\lambda;\mu)}^2}{\alpha T}}$$
Substituting this into \eqref{eq:vqq} and \eqref{eq:vqz} proves \eqref{eq:Dav} and \eqref{eq:Dav2}.
\end{proof}

\

Of independent interest, the inequality \eqref{eq:independent-interest} ensures the expectations of $\norm{\pi(t)-\pi(t-1)}_1^2$ tend to be small. For example, fix $\epsilon>0$ and define $V=1/\epsilon$ and $\alpha=1/\epsilon^2$. Then \eqref{eq:independent-interest} implies 
$$ \frac{1}{T}\sum_{t=0}^T\expect{\norm{\pi(t)-\pi(t-1)}_1^2} \leq O(\epsilon^2) \quad \forall T \geq 1/\epsilon^2$$

\section*{Appendix B -- Proof of Theorems \ref{thm:converse} and \ref{thm:achievability}} 
\setcounter{subsection}{0}

This appendix develops several lemmas and them uses them to prove Theorems  \ref{thm:converse} and \ref{thm:achievability}.  Recall that for each $i \in \script{S}$, the set $\Gamma_i\subseteq\mathbb{R}^{|\script{I}|}$ is defined in \eqref{eq:gamma-def}.

\subsection{Properties of $\Gamma_i$}

\begin{lem} \label{lem:convex} For each $i \in \script{S}$, the set $\Gamma_i$ is bounded and convex. Its closure $\overline{\Gamma}_i$ is compact and convex. 
\end{lem} 

\begin{proof} 
Set $\Gamma_i$ is bounded because all components of $g_i$ in \eqref{eq:g} are bounded. 
To show $\Gamma_i$ is convex, fix $\gamma_1, \gamma_2 \in \Gamma_i$ and $\theta \in (0,1)$. Define $(W,U)=(W(0),U(0))$. By definition of $\Gamma_i$ we have 
\begin{align*}
\gamma_1 &=\expect{g_i(W, \alpha(W, U))} \\
\gamma_2 &= \expect{g_i(W, \beta(W, U))} 
\end{align*}
for some functions $\alpha, \beta \in \script{C}_i$. Define a new function $\phi \in \script{C}_i$ by 
$$ \phi(w,u) = \left\{\begin{array}{cc}
\alpha(w,\frac{u}{\theta}) & \mbox{if $u \in [0, \theta]$} \\
\beta(w,\frac{u-\theta}{1-\theta}) & \mbox{if $u\in (\theta, 1]$}
\end{array}\right.$$
for $w \in \script{W}$ and $u\in [0,1]$. 
By definition of $\Gamma_i$ it holds that $\expect{g_i(W, \phi(W,U))} \in \Gamma$. Since $U \sim \mbox{Unif}[0,1]$ we have
\begin{align*}
\expect{g_i(W,\phi(W,U))} &= \expect{g_i(W,\phi(W,U))|U\leq \theta}\theta + \expect{g_i(W, \phi(W,U))|U>\theta}(1-\theta) \\
&= \expect{g_i(W,\alpha(W,U/\theta))|U\leq \theta}\theta + \expect{g_i(W, \beta(W,\mbox{$\frac{U-\theta}{1-\theta}$}))|U>\theta}(1-\theta) \\
&\overset{(a)}{=} \expect{g_i(W,\alpha(W,U))}\theta + \expect{g_i(W, \beta(W,U))}(1-\theta) \\
&=\theta \gamma_1 + (1-\theta)\gamma_2
\end{align*}
where equality (a) uses the fact that $W$ and $U\sim \mbox{Unif}[0,1]$ are independent and so 
\begin{itemize} 
\item The conditional distribution of $(W, U/\theta)$, given $\{U\leq \theta\}$, is the same as the unconditional distribution of $(W,U)$.  
\item The conditional distribution of $(W, \mbox{$\frac{U-\theta}{1-\theta}$})$, given $\{U>\theta\}$,  is the same as the unconditional distribution of $(W,U)$.  
\end{itemize} 
Thus,  $\theta \gamma_1 + (1-\theta)\gamma_2 \in \Gamma_i$, so  $\Gamma_i$ is convex.  As $\Gamma_i$ is bounded and convex, $\overline{\Gamma}$ is compact and convex. 
\end{proof} 

\

\begin{lem} \label{lem:in-gamma0} Fix $i \in \script{S}$. Let $W:\Omega\rightarrow\script{W}$ be a random element with the same distribution as $W(0)$. Let $A_i:\Omega\rightarrow\script{A}_i$ be any random element (possibly dependent on $W$). Then 
$\expect{g_i(W,A)} \in \Gamma_i$.
\end{lem} 

\begin{proof} 
Without loss of generality, assume there is a $V\sim \mbox{Unif}[0,1]$ that is independent of $(W, A_i)$ (if this is not true, extend the probability space using standard product space concepts and note that this does not change any expectations). Since $A_i:\Omega\rightarrow \script{A}_i$ and $(\script{A}_i,\script{G}_i)$ is a Borel space, we have by Lemma \ref{lem:coupling} (in Appendix C) that 
$A_i = \alpha_i(W, U)$ for some random variable $U\sim Unif[0,1]$ that is independent of $W$ and some 
measurable function $\alpha_i:\script{W}\times [0,1]\rightarrow\script{A}_i$. That is, $\alpha_i \in \script{C}_i$.  Then 
$$\expect{g_i(W, A_i)} = \expect{g_i(W, \alpha_i(W,U))}  \in \Gamma_i$$
where the final inclusion holds by definition of $\Gamma_i$ in \eqref{eq:gamma-def}. 
\end{proof} 

\

\begin{lem} \label{lem:in-gamma} Fix $i \in \script{S}$. Let $W:\Omega\rightarrow\script{W}$ be a random element with the same distribution as $W(0)$. Let $X:\Omega\rightarrow[0,\infty)$ be a  nonnegative random variable that is independent of $W$ 
and that has finite $\expect{X}$. Let $A_i:\Omega\rightarrow\script{A}_i$ be a random element (possibly dependent on $(W,X)$). 
Then 
\begin{equation} \label{eq:in-X-gamma}
\expect{Xg_i(W,A_i)} \in \expect{X}\Gamma_i  
\end{equation} 
where the right-hand-side uses the Minkowski set scaling $\expect{X}\Gamma_i=\{\expect{X}\gamma : \gamma \in \Gamma_i\}$.
\end{lem} 

\begin{proof} 
 If $\expect{X}=0$ then $X=0$ almost surely (recall $X$ is nonnegative) and so \eqref{eq:in-X-gamma} reduces to the trivially true statement $0 \in \{0\}$. Suppose $\expect{X}>0$. 
Let $Y$ be a random variable that has distribution 
$$P[Y \in B] = \frac{\expect{X1_{\{X\in B\}}}}{\expect{X}} \quad\forall B \in \script{B}(\mathbb{R})$$
For any positive integer $d$ and any bounded and measurable function $f:\mathbb{R}\rightarrow\mathbb{R}^d$ we have (see Lemma \ref{lem:skew}):
\begin{equation} \label{eq:Yav}
\expect{f(Y)} = \frac{\expect{Xf(X)}}{\expect{X}}
\end{equation} 
As in the previous lemma, without loss of generality assume there is a $V\sim \mbox{Unif}[0,1]$ that is independent of $(W, A_i, X)$ (else, extend the probability space). 
The representation result of Lemma \ref{lem:coupling} implies $A_i=\theta(W,X, U)$ for some $U\sim \mbox{Unif}[0,1]$ that is independent of $(W,X)$ and some function $\theta \in \script{C}_i$. Define the bounded and measurable function $f:\mathbb{R}\rightarrow\mathbb{R}^{|\script{I}|}$ by 
$$ f(x) = \expect{g_i(W, \theta(W, x, U))} \quad \forall x\geq 0$$
The definition of $\Gamma_i$ in \eqref{eq:gamma-def} implies $f(x)\in\Gamma_i$ for all $x\geq 0$. Since $X$ is independent of $(W,U)$ it holds that (almost surely)
$$ \expect{g_i(W,\theta(W,X,U))|X} = f(X)$$
Multiplying both sides of the above equality by $X$ and taking expectations gives
\begin{equation} \label{eq:andso}
\expect{Xg_i(W,\theta(W,X,U))} = \expect{Xf(X)} 
\end{equation} 
Since $\expect{X}>0$ and $A_i=\theta(W,X,U)$ we have 
\begin{align*}
\frac{\expect{Xg_i(W,A_i)}}{\expect{X}} &= \frac{\expect{Xg_i(W, \theta(W,X,U))}}{\expect{X}} \\
&\overset{(a)}{=}\frac{\expect{Xf(X)}}{\expect{X}} \\
&\overset{(b)}{=}\expect{f(Y)} 
\end{align*}
where (a) holds by \eqref{eq:andso}; (b) holds by \eqref{eq:Yav}.  We know $f(Y)$ is a random vector that takes values in the bounded and convex set $\Gamma_i$ surely, so $\expect{f(Y)} \in \Gamma_i$ (see Lemma \ref{lem:EX-in-A}), which  proves \eqref{eq:in-X-gamma}.
\end{proof} 

\subsection{Conditional expectations} 

Define: 
\begin{align*}
Y_{i,l}(t) &= c_{i,l}(W(t),A_i(t))\\
1_i(t) &= 1_{\{S(t)=i\}}\\
1_{i,j}(t) &= 1_{\{S(t)=i\}}1_{\{S(t+1)=j\}} 
\end{align*}
for $t \in\{0, 1, 2, \ldots\}$, $i,j \in \script{S}$, and $l \in \{0, \ldots, k\}$. 


\begin{lem} \label{lem:causal-in-gamma} Suppose $\{A(t)\}_{t=0}^{\infty}$ is a sequence of causal and measurable actions (so \eqref{eq:general-action} is satisfied).  For each $i \in \script{S}$ and $t \in \{0, 1, 2, \ldots\}$, there is a random vector $\gamma_i(t)\in \overline{\Gamma}_i$ such that, with probability 1, 
\begin{align}
&\expect{(1_i(t)Y_{i,l}(t); 1_{i,j}(t))_{(l,j)\in\script{I}}|H(t), S(t)} =1_{\{S(t)=i\}}\gamma_i(t)\label{eq:c-lem} 
\end{align}
\end{lem} 
\begin{proof} Fix $i \in \script{S}$ and $t \in \{0, 1, 2, \ldots\}$. 
Define
\begin{equation} \label{eq:Xi} 
X_i(t) = 1_i(t)(c_{i,l}(W(t), A_i(t)); 1_{\{S(t+1)=j\}})_{(l,j)\in\script{I}}
\end{equation} 
We want to show that with probability 1
\begin{equation} \label{eq:suffice0}
\expect{X_i(t)|H(t),S(t)} \in 1_{\{S(t)=i\}}\overline{\Gamma}_i
\end{equation} 
where the right-hand-side is the set that contains only the zero vector  if $S(t)\neq i$. 
For $i,j\in\script{S}$ we have (almost surely) 
$$ \expect{1_{\{S(t)=i\}}1_{\{S(t+1)=j\}}|H(t), S(t), W(t), A(t)} =  1_{\{S(t)=i\}}p_{i,j}(W(t),A_i(t))$$
The tower property of conditional expectations ensures (almost surely) 
$$\expect{1_{\{S(t)=i\}}1_{\{S(t+1)=j\}}|H(t),S(t)} = 1_{\{S(t)=i\}}\expect{p_{i,j}(W(t),A_i(t))|H(t),S(t)}$$
Therefore, with $g_i$ defined by \eqref{eq:g}, we have (almost surely) 
$$\expect{X_i(t)|H(t), S(t)} = 1_{\{S(t)=i\}}\expect{g_i(W(t), A_i(t))|H(t), S(t)}$$
Let $Z=\expect{g_i(W(t),A_i(t))|H(t),S(t)}$ be any version of the conditional expectation. 
Then (almost surely) 
$$\expect{X_i(t)|H(t),S(t)} = 1_{\{S(t)=i\}}Z$$ 
Substituting this into  \eqref{eq:suffice0}, it suffices to show that with probability 1
$$ 1_{\{S(t)=i\}}Z \in 1_{\{S(t)=i\}}\overline{\Gamma}_i$$
It suffices to show 
\begin{equation}\label{eq:sufficex}
P[\{S(t)\neq i\}\cup \{Z \in \overline{\Gamma}_i\}]=1 
\end{equation}

Since $\overline{\Gamma}_i$ is a compact and convex subset of $\mathbb{R}^{|\script{I}|}$,  it is  \emph{constructible}, meaning it is the intersection of a countable number of closed half-spaces \cite{constructible-book}: 
\begin{equation}\label{eq:constructible}
\overline{\Gamma}_i= \cap_{m=1}^{\infty} \{x \in \mathbb{R}^{|\script{I}|} : a_m^{\top}x \leq b_m\}
 \end{equation} 
for some vectors $a_m\in\mathbb{R}^{|\script{I}|}$ and scalars $b_m$ for $m \in \mathbb{N}$. Thus, 
$$ \{Z \in \overline{\Gamma}_i\} \iff \cap_{m=1}^{\infty} \{a_m^{\top} Z \leq b_m\}$$
To show \eqref{eq:sufficex}, it suffices to show 
\begin{equation} \label{eq:suffice-m}
P[\{S(t) \neq i\} \cup \{a_m^{\top}Z \leq b_m\}]=1 \quad \forall m \in \mathbb{N}
\end{equation} 
Indeed, if each event in a countable sequence of events has probability 1, their countable intersection also has  probability 1.  

Suppose \eqref{eq:suffice-m} fails (we reach a contradiction). Then there is a $m \in \mathbb{N}$ such that  
$$P[\{S(t)=i\}\cap \{a_m^{\top}Z > b_m\}]>0$$
By continuity of probability, there must be an $\epsilon>0$ such that 
$$P[\{S(t)=i\}\cap \{a_m^{\top}Z \geq  b_m+\epsilon\}]>0$$
Define the event
\begin{equation} \label{eq:F}
F = \{S(t)=i\}\cap \{a_m^{\top}Z\geq b_m+\epsilon\} 
\end{equation} 
and note that $P[F]>0$. 

Since $Z$ is a conditional expectation given $(H(t),S(t))$, it holds that $Z$ is a measurable function
of $(H(t), S(t))$, so $F$ is in the sigma algebra generated by $(H(t),S(t))$. Thus, the definition of conditional expectation implies 
$$ \expect{1_FZ} = \expect{1_Fg_i(W(t), A_i(t))}$$
By linearity of expectation 
$$ \expect{1_F(a_m^{\top}Z- b_m)} = a_m^{\top}\expect{1_Fg_i(W(t),A_i(t))}  -b_m\expect{1_F}$$
On the other hand, definition of $F$ gives $\{1_F=1\}\implies \{a_m^{\top}Z - b_m\geq \epsilon\}$ and so 
$$ \epsilon \expect{1_F} \leq a_m^{\top}\expect{1_Fg_i(W(t),A_i(t))}  -b_m\expect{1_F}$$

Applying the representation lemma (Lemma \ref{lem:coupling} of Appendix C) 
to represent $A_i(t)$ in terms of the 
random element $(W(t), S(t), 1_F)$ (using the existence of the independent 
$V(-1)\sim \mbox{Unif}[0,1]$ to enable the representation) 
 yields 
$$ A_i(t) = \beta(W(t),S(t), 1_F, U)$$
for some measurable function $\beta$ and some 
random variable $U\sim \mbox{Unif}[0,1]$ that is independent of $(W(t),S(t),1_F)$. Since $W(t)$ is independent of $(S(t), 1_F)$, it holds that $(W(t),U)$ is independent of $(S(t),1_F)$.  Thus
\begin{align*}
\epsilon \expect{1_F} &\leq a_m^{\top}\expect{1_Fg_i(W(t),\beta(W(t),S(t),1_F,U))}  -b_m\expect{1_F}\nonumber\\
&= a_m^{\top}\expect{1_Fg_i(W(t),\beta(W(t),i,1,U))}  -b_m\expect{1_F}\nonumber\\
&= a_m^{\top}\expect{1_F}\expect{g_i(W(t),\beta(W(t),i,1,U))} - b_m\expect{1_F}
\end{align*}
where the first equality holds because the definition of $F$ in \eqref{eq:F} means $\{1_F=1\} \implies \{S(t)=i\}$; the second equality holds because $1_F$ is independent of $(W(t), U)$.  Dividing the above by $P[F]$ and using $\expect{1_F}=P[F]$ gives
\begin{equation} \label{eq:suby}
\epsilon \leq a_m^{\top}\expect{g_i(W(t),\beta(W(t),i,1,U))} - b_m
\end{equation} 
Define  $\alpha_i\in\script{C}_i$ by 
$$\alpha_i(w,u) = \beta(w,i,1,u) \quad \forall (w,u)\in\script{W}\times [0,1]$$
Define 
$$ y = \expect{g_i(W(t),\alpha_i(W(t),U))}$$
It follows that $y \in \Gamma_i$ by definition of $\Gamma_i$ in \eqref{eq:gamma-def}. Substituting $y$ into \eqref{eq:suby}  gives
$$\epsilon  \leq a_m^{\top}y-b_m$$
So \eqref{eq:constructible} implies $y \notin \overline{\Gamma}_i$, contradicting $y \in \Gamma_i$. This completes the proof. 
\end{proof} 

\subsection{Proof of Lemma \ref{lem:exist}}

Suppose the deterministic problem is feasible. Compactness of the set of decision variables that satisfy \eqref{eq:det4}-\eqref{eq:det5} and continuity of the functions in \eqref{eq:det1}-\eqref{eq:det3} together imply that at least one optimal solution exists.  Every optimal solution must satisfy \eqref{eq:equal-c0} and must have a resulting $(p_{i,j})$ that is a transition probability matrix with disjoint communicating classes $\script{S}^{(1)}, \ldots, \script{S}^{(m)}$ for some $m\in\{1, \ldots, n\}$.  It can be shown that there is $c \in \{1, \ldots, m\}$ such that all constraints remain satisfied, and equality \eqref{eq:equal-c0} is maintained, when the optimal solution is modified by changing $\pi$ to have support only over states in $\script{S}^{(c)}$. This completes the proof. 

\subsection{Proof of Theorem \ref{thm:converse}}

Suppose the stochastic problem \eqref{eq:p1}-\eqref{eq:p3} is feasible. 
We first show \eqref{eq:fatou} is directly implied by \eqref{eq:converse}. Since $|C_0(t)|\leq c_{max}$ surely, the process $C_0(t)+c_{max}$ is nonnegative and Fatou's lemma implies 
$$\expect{\liminf_{T\rightarrow\infty} \frac{1}{T}\sum_{t=0}^{T-1}(C_0(t)+c_{max})} \leq \liminf_{T\rightarrow\infty} \frac{1}{T}\sum_{t=0}^{T-1}\expect{C_0(t)+c_{max}}$$
Therefore
$$ \expect{\liminf_{T\rightarrow\infty} \frac{1}{T}\sum_{t=0}^{T-1}C_0(t)} \leq \liminf_{T\rightarrow\infty} \frac{1}{T}\sum_{t=0}^{T-1}\expect{C_0(t)}$$
If \eqref{eq:converse} holds then the left-hand-side of the above inequality is greater than or equal to $c_0^*$, which implies \eqref{eq:fatou}. 

It remains to show that \eqref{eq:converse} holds and that the deterministic problem is feasible. 
Since the stochastic problem \eqref{eq:p1}-\eqref{eq:p3}  is feasible, there are actions  
$\{A(t)\}_{t=0}^{\infty}$ of the form \eqref{eq:general-action}  that produce costs $\{C_l(t)\}_{t=0}^{\infty}$ that 
satisfy the constraints  \eqref{eq:p2}-\eqref{eq:p3} almost surely.  For these actions, let  $\tilde{\Omega}$ be the set of all outcomes $\omega \in \Omega$ such that for all $l \in \{1, \ldots, k\}$ and all $i \in \script{S}$:
\begin{align}
&\limsup_{T\rightarrow\infty}\frac{1}{T}\sum_{t=0}^{T-1}C_l(t) \leq 0\label{eq:sample1} \\
&\expect{X_i(t)|H(t),S(t)} = 1_{\{S(t)=i\}}\gamma_i(t) \label{eq:sample2} \\
&\lim_{T\rightarrow\infty} \frac{1}{T}\sum_{t=0}^{T-1}\left(X_i(t)-\expect{X_i(t)|H(t),S(t)}\right)=0\label{eq:sample3}
\end{align}
for some $\gamma_i(t)\in \overline{\Gamma}_i$, where $X_i(t)$ is defined in \eqref{eq:Xi}. We know \eqref{eq:sample1} holds with probability 1; Lemma \ref{lem:causal-in-gamma} ensures \eqref{eq:sample2} holds with probability 1; Lemma \ref{lem:LLN} (Appendix C) ensures \eqref{eq:sample3} holds with probability 1. Thus, $P[\tilde{\Omega}]=1$.

For the rest of the proof we consider the sample path for a particular outcome $\omega \in \tilde{\Omega}$, so that \eqref{eq:sample1}-\eqref{eq:sample3} hold. For simplicity of notation we continue to write values such as $C_l(t)$, rather than $C_l(t)(\omega)$.  For this $\omega \in \tilde{\Omega}$ define 
\begin{equation} \label{eq:c0}
c_0 = \liminf_{T\rightarrow\infty}\frac{1}{T}\sum_{t=0}^{T-1}C_0(t)
\end{equation} 
All sample path values are bounded, so there is a subsequence $\{T_m\}_{m=1}^{\infty}$ of increasing positive integers such that
\begin{align*}
&\lim_{m\rightarrow\infty} \frac{1}{T_m}\sum_{t=0}^{T_m-1}C_l(t) = c_l \quad \forall l \in \{0, \ldots, k\}\\
&\lim_{m\rightarrow\infty}\frac{1}{T_m}\sum_{t=0}^{T_m-1}1_{\{S(t)=i\}}\gamma_i(t) = x_i \quad \forall i \in \script{S}\\
&\lim_{m\rightarrow\infty}\frac{1}{T_m}\sum_{t=0}^{T_m-1}1_{\{S(t)=i\}} = \pi_i\quad \forall i \in \script{S}
\end{align*}
for some $(\pi_1, \ldots, \pi_n)\in\script{P}$, some $x_i \in \mathbb{R}^{|\script{I}|}$ for $i \in \script{S}$,  some $c_i\leq 0$ for $i \in \{1, \ldots, k\}$, and $c_0$ defined by \eqref{eq:c0}.

We first claim that for $i\in\script{S}$
\begin{equation} \label{eq:claim1} 
x_i=\pi_i  ((c_{i,l});(p_{i,j}))
\end{equation} 
for some $((c_{i,l});(p_{i,j})) \in \overline{\Gamma}_i$.  If $\pi_i=0$ this holds trivially. Suppose $\pi_i>0$. Then for all sufficiently large $m$  
$$ \frac{\sum_{t=0}^{T_m-1}1_{\{S(t)=i\}}\gamma_i(t)}{\sum_{t=0}^{T_m-1}1_{\{S(t)=i\}}} \in \overline{\Gamma}_i$$
because the convex combination of points in $\overline{\Gamma}_i$ is again in $\overline{\Gamma}_i$. Multiplying both sides by $\pi_i$ and taking $m\rightarrow\infty$ yields $x_i \in \pi_i \overline{\Gamma}_i$ (note that $\pi_i\overline{\Gamma}_i$ is a compact set), which proves  \eqref{eq:claim1}. Then for each $l \in \{0, \ldots, k\}$ and $i,j\in \script{S}$ we have 
\begin{align*}
&c_l= \lim_{m\rightarrow\infty}\frac{1}{T_m}\sum_{t=0}^{T_m-1}\sum_{i\in\script{S}} 1_i(t)C_{i,l}(t) = \sum_{i\in\script{S}}\pi_ic_{i,l}\\
&\lim_{m\rightarrow\infty} \frac{1}{T_m}\sum_{t=0}^{T_m-1}1_{i,j}(t) = \pi_ip_{i,j}
\end{align*}
Since $c_l\leq 0$ for $l \in \{1, \ldots, k\}$, it follows that the $\pi_i$, $c_{i,l}$, $p_{i,j}$ variables satisfy the constraints 
\eqref{eq:det3},\eqref{eq:det4},\eqref{eq:det5}, and yield objective value $c_0$ in \eqref{eq:det1}.  To show the deterministic problem \eqref{eq:det1}-\eqref{eq:det5} is feasible and that  $c_0\geq c_0^*$, where $c_0^*$ is the minimum objective value for the problem, it suffices to show the $\pi_i, p_{i,j}$ satisfy \eqref{eq:det2}.

Fix $j \in \script{S}$. The number of times $S(t)=j$ is within 1 of the number of times we transition into $j$, so that 
\begin{equation} \label{eq:within1} 
\left|\sum_{t=0}^{T_m} 1_{\{S(t)=j\}} - \sum_{t=0}^{T_m-1}\sum_{i\in\script{S}} 1_{i,j}(t)\right|  \leq 1 \quad \forall m \in \{1, 2, 3, \ldots\}
\end{equation} 
Taking $m\rightarrow\infty$ gives 
\begin{equation} \label{eq:stat}
\pi_j=\sum_{i\in\script{S}} \pi_ip_{i,j}
\end{equation} 
so that  \eqref{eq:det2} holds. Overall, $\tilde{\Omega}$ is a set that satisfies $P[\tilde{\Omega}]=1$, and each outcome $\omega \in \tilde{\Omega}$ yields a sample path that satisfies
$$ c_0^*\leq c_0=\liminf_{T\rightarrow\infty}\frac{1}{T}\sum_{t=0}^{T-1}C_0(t)$$
which completes the proof. 

\subsection{Proof of Theorem \ref{thm:achievability}} 

If the deterministic problem \eqref{eq:det1}-\eqref{eq:det5} 
is feasible, Lemma \ref{lem:exist} ensures there is an optimal solution 
$((\pi_i); (c_{i,l}); (p_{i,j}))$ for which there is a nonempty set  $\script{S}'\subseteq\script{S}$ such that $\pi_i=0$ if $i \notin \script{S}'$, and for which the submatrix 
$(p_{i,j})_{i,j\in\script{S}'}$ is irreducible over $\script{S}'$.  By definition of an optimal solution to \eqref{eq:det1}-\eqref{eq:det5}  we have 
\begin{align}
&\sum_{i \in \script{S}} \pi_ic_{i,0} = c_0^*\label{eq:also1}\\
&\pi_j=\sum_{i\in\script{S}} \pi_ip_{i,j}  \quad \forall j \in \script{S}\label{eq:also2}\\
&\sum_{i \in \script{S}} \pi_ic_{i,l} \leq 0 \quad \forall l \in \{1, \ldots, k\} \label{eq:also3} \\
&\pi \in \script{P} \quad , \quad ((c_{i,l}); (p_{i,j})) \in \overline{\Gamma}_i \quad \forall i \in \script{S}\label{eq:also4} 
\end{align}

Define $W=W(0)$ and $U=U(0)$. Fix $i \in \script{S}$. The set $\Gamma_i$ is closed, and so $\overline{\Gamma}_i=\Gamma_i$ and  \eqref{eq:also4} implies 
$$((c_{i,l}); (p_{i,j}))_{(l,j)\in\script{I}} \in \Gamma_i$$
By definition of $\Gamma_i$ in \eqref{eq:gamma-def}, there is a function $\alpha_i \in \script{C}_i$ that satisfies 
$$ \expect{g_i(W,\alpha_i(W,U))} = ((c_{i,l}); (p_{i,j}))_{(j,l)\in\script{I}} $$
where $g_i$ is defined in \eqref{eq:g}. 
Fix $s_0\in\script{S}'$ and define $S(0)=s_0$ surely. Use the memoryless actions $A_i(t)=\alpha_i(W(t),U(t))$ for $i \in \script{S}$ and $t \in \{0, 1, 2, \ldots\}$. 
Define $X_i(t)$ for $i \in \script{S}$ by 
$$ X_i(t) = g_i(W(t), A_i(t))\quad \forall t \in \{0, 1, 2, \ldots\}$$
Then for each $i \in \script{S}$,  $\{X_i(t)\}_{t=0}^{\infty}$ are i.i.d. random vectors. 
Moreover, these vectors depend only on $\{(W(t), U(t))\}_{t=0}^{\infty}$, so they  are independent of $\{S(t)\}_{t=0}^{\infty}$. Also, each $X_i(t)$ is independent of history $H(t)$. Definition of $g_i$ in \eqref{eq:g} yields for all $t$
\begin{align*}
&\expect{C_l(t)|S(t)=i}=\expect{c_{i,l}(W(t),A_i(t))} = c_{i,l} \\
&P[S(t+1)=j|S(t)=i]=\expect{p_{i,j}(W(t),A_i(t))} = p_{i,j} 
\end{align*}
for all $i,j \in \script{S}$ and $l \in \{0, \ldots, k\}$. So $\{S(t)\}_{t=0}^{\infty}$ is a Markov chain with initial state $S(0)=s_0\in \script{S}'$ and transition probabilities $(p_{i,j})$. Its state is always in $\script{S}'$. The time average fraction of time in each state converges to the unique solution to \eqref{eq:also2} that corresponds to $\script{S}'$, which is the vector $(\pi_i)_{i\in\script{S}'}$ itself. Thus, the time average costs also satisfy \eqref{eq:also1}, \eqref{eq:also3}, 
completing the proof.

\section*{Appendix C -- Probability tools} 
\setcounter{subsection}{0}

\subsection{Representation of a random element that takes values on a Borel space} 

The following is a simple extension of a coupling result of Kallenberg that holds for any random element $X$ that takes values on a Borel space (see Proposition 5.13 in \cite{kallenberg}).  It ensures that for any other (possibly dependent) random element $S$, we can represent $X$ in terms of $S$ and an independent $\mbox{Unif}[0,1]$ random variable. 
The lemma extends the Kallenberg result from \emph{almost surely} to \emph{surely}.  While our proof uses elementary steps, we are unaware of a similar statement in the literature.  
The lemma assumes existence of a \emph{randomization variable} 
$V \sim \mbox{Unif}[0,1]$ that lives on the same probability space as $(X,S)$, but is independent of $(X, S)$. As described in \cite{kallenberg} (just before Theorem 5.10 there), existence of $V$ is a mild assumption that, if needed, can be guaranteed to hold by extending the probability space using standard product space concepts. 

\

\begin{lem} \label{lem:coupling} (Representation of a random element $X$) Fix a probability space $(\Omega, \script{F}, P)$. Suppose 
\begin{itemize} 
\item $X:\Omega\rightarrow D_X$ is a random element that takes values on some Borel space $(D_X,\script{G}_X)$.
\item  $S:\Omega\rightarrow D_S$ is a random element that takes values on some measurable space $(D_S, \script{G}_S)$ (not necessarily a Borel space). 
\item  $V:\Omega\rightarrow [0,1]$ is a $\mbox{Unif}[0,1]$ random variable that is independent of $(X,S)$.
\end{itemize} 
Then we surely have 
$$X=f(S,U)$$ 
for some measurable function $f:D_S\times [0,1]\rightarrow D_X$ and some random variable $U \sim \mbox{Unif}[0,1]$ that is independent of $S$.  In particular, $U$ is a measurable function of $(X,S,V)$. 
\end{lem}

\

\begin{proof}  
Proposition 5.13 in \cite{kallenberg} shows $X=g(S,Y)$ \emph{almost surely} for some measurable function $g:D_S\times [0,1]\rightarrow D_X$ and some random variable $Y \sim \mbox{Unif}[0,1]$ that is independent of $S$ and that is a measurable function of $(X,S,V)$. 
Here we extend this to surely. Recall that an \emph{isomorphism} between two measurable spaces $(A_1, \script{F}_1)$ and $(A_2, \script{F}_2)$ is a measurable bijective function $\psi:A_1\rightarrow A_2$ with a measurable inverse. 
We first argue there is a Borel measurable set $C\subseteq [0,1]$ that has Borel measure 0, and an isomorphism $\psi:D_X\rightarrow C$. To see this, first suppose $D_X$ is uncountably infinite. Let $C$ be any uncountably infinite Borel measurable subset of $[0,1]$ that has Borel measure $0$ (such as a Cantor set). Theorem 3.3.13 in \cite{srivastava-borel} ensures there is an isomorphism between any two uncountably infinite Borel spaces, so the desired $\psi:D_X\rightarrow C$ exists. In the opposite case when $D_X$ is  finite or countably infinite,  it is easy to construct a Borel measurable subset $C \subseteq [0,1]$ with the same cardinality as $D_X$, so  $C$ has Borel measure zero and the desired  $\psi$ again exists.

With the measure-zero set $C\subseteq [0,1]$ and the isomorphism $\psi:D_X\rightarrow C$ in hand, define  random variable $U:\Omega\rightarrow[0,1]$ by 
$$ U=\left\{\begin{array}{cc}
Y & \mbox{ if $Y \notin C$ and $X= g(S,Y)$} \\
\psi(X) & \mbox{ else} 
\end{array}\right.$$
Since $P[Y \notin C]=1$ and $P[X=g(S,Y)]=1$, it holds that $P[U=Y]=1$. Therefore, $U$ has the same distributional properties as $Y$, specifically, $U \sim \mbox{Unif}[0,1]$ and $U$ is independent of $S$. By definition of $U$ we have 
\begin{align}
&\{U \in C\} \implies \{U=\psi(X)\}\label{eq:iff1}\\ 
&\{U \notin C\} \implies \{U=Y\} \cap \{X=g(S,Y)\}\label{eq:iff2} 
\end{align}
Define the measurable function $f:D_S\times [0,1]\rightarrow D_X$ by 
$$ f(s,u) = \left\{\begin{array}{cc}
g(s,u) & \mbox{if $u \notin C$} \\
\psi^{-1}(u) &\mbox{if $u \in C$} 
\end{array}\right.$$ 
To show $X=f(S,U)$ surely, observe that if $U\in C$ then by definition of $f$: 
$$ f(S,U)=\psi^{-1}(U) \overset{(a)}{=} \psi^{-1}(\psi(X)) = X$$
where equality (a) holds by \eqref{eq:iff1}.  On the other hand, if $U\notin C$ then by definition of $f$: 
$$ f(S,U)= g(S,U) \overset{(b)}{=}g(S,Y) \overset{(c)}{=} X$$
where equalities (b) and (c) hold because \eqref{eq:iff2} implies $U=Y$ and $X=g(S,Y)$. 
\end{proof} 

\subsection{Variation on the law of large numbers} 


Let  $(\Omega, \script{F}, P)$ be a probability space.  Fix $d\in\mathbb{N}$ and define $\norm{x}$ as the standard Euclidean norm in $\mathbb{R}^d$, so 
$$ \norm{x}^2 = \sum_{i=1}^d x_i^2 \quad \forall x \in \mathbb{R}^d$$
The following lemma follows directly from a result in  \cite{chow-lln}.  It concerns general random vectors $\{X_n\}_{n=1}^{\infty}$, possibly dependent and having different distributions. While it uses a filtration $\{\script{F}_n\}_{n=1}^{\infty}$ with certain properties, an important example to keep in mind is $\script{F}_1=\{\Omega, \phi\}$ and $\script{F}_n =\sigma(X_1, \ldots, X_{n-1})$ for $n \geq 2$.

\

\begin{lem} \label{lem:LLN} 
Let $\{X_n\}_{n=1}^{\infty}$ be  random vectors that take values in $\mathbb{R}^d$. Let $\{\script{F}_n\}_{n=1}^{\infty}$ be a sequence of sigma algebras on $\Omega$ such that $\script{F}_n\subseteq\script{F}_{n+1}$ for all $n \in \mathbb{N}$. Suppose that:
\begin{itemize} 
\item $(X_1, \ldots, X_{n-1})$ is $\script{F}_n$-measurable for all $n \geq 2$.
\item $\sum_{n=1}^{\infty} \frac{\expect{\norm{X_n}^2}}{n^2}<\infty$  
\end{itemize} 
Then
$$\lim_{n\rightarrow\infty} \frac{1}{n}\sum_{i=1}^n\left(X_i- \expect{X_i|\script{F}_i}\right) = 0 \quad \mbox{(almost surely)} $$
\end{lem} 

\

\begin{proof} 
This follows directly from the law of large numbers for martingale differences in \cite{chow-lln} by defining $Y_i=X_i-\expect{X_i|\script{F}_i}$ and observing that $\expect{Y_i|\script{F}_i}=0$ for all $i\in\mathbb{N}$.
\end{proof} 

\subsection{Finite expectations cannot leave convex sets} 

Let  $(\Omega, \script{F}, P)$ be a probability space. Fix $n$ as a positive integer. We say random vector $X:\Omega\rightarrow\mathbb{R}^n$ has \emph{finite expectation} (equivalently, $\expect{X}$ is finite),  if and only if $\expect{X} \in \mathbb{R}^n$, where $\expect{X}=(\expect{X_1}, \ldots, \expect{X_n})$. The following lemma is generally accepted as true, although we cannot find a complete proof in the literature.  We provide our own proof below.  The challenge is that the convex set $A\subseteq\mathbb{R}^n$ is arbitrary and is not necessarily closed. In fact, $A$ is not necessarily Borel measurable, 
although we assume $\{X \in A\}$ is an event that has probability 1 (an example is when $X \in A$ surely).\footnote{If $X:\Omega \rightarrow\mathbb{R}^n$ is a random vector and $A$ is a Borel measurable subset of $\mathbb{R}^n$ then $\{X \in A\}$ is \emph{always} an event. However, if $A$ is not Borel measurable then $\{X \in A\}$ may or may not be an event. An example convex set $A\subseteq \mathbb{R}^2$ that is not Borel measurable is $A = \{(x,y)\in \mathbb{R}^2: x^2+y^2<1\} \cup S$ where $S$ is any nonBorel subset of the unit circle $\{(x,y)\in \mathbb{R}^2: x^2+y^2=1\}$. Such a set $S$ exists under the axiom of choice. A trivial example random vector $X:\Omega\rightarrow\mathbb{R}^2$ for which $\{X \in A\}$ is an event is the always-zero random vector, so $\{X \in A\} = \{X=(0,0)\}=\Omega$.}

\

\begin{lem} \label{lem:EX-in-A} Let $A\subseteq\mathbb{R}^n$ be a convex set. Let $X:\Omega\rightarrow\mathbb{R}^n$ be a random vector such that $\{X \in A\}$ is an event and 
$P[X \in A]=1$. If $\expect{X}$ is finite then $\expect{X} \in A$. 
\end{lem} 

\

\begin{proof} Suppose  $P[X \in A]=1$ and $\expect{X}$ is finite. 
Without loss of generality it suffices to assume $\expect{X}=0$ (else, define $\tilde{X}=X-\expect{X}$ and define the shifted set $\tilde{A}=A - \expect{X}$). We want to show $0 \in A$. Suppose $0\notin A$ (we reach a contradiction). Let $d\in\{0, 1, \ldots, n\}$ be the smallest integer for which there is a $d$-dimensional linear subspace $L\subseteq\mathbb{R}^n$ such that $P[X \in L]=1$ (note that $\mathbb{R}^n$ is itself a $n$-dimensional linear subspace such that $P[X \in \mathbb{R}^n]=1$, so $d\leq n$). 

First suppose $d=0$. Then $L=\{0\}$ and $P[X=0]=1$. By assumption, we also have $P[X \in A]=1$.  The two probability-1 events $\{X=0\}$ and $\{X \in A\}$ cannot be disjoint (since then the union would have probability 2). Thus, $0\in A$, a contradiction. 

Now suppose $d\geq 1$. Let $L$ be the corresponding $d$-dimensional subspace. Since $P[X \in L]=1$ and $P[X \in A]=1$, we know $\{X \in A \cap L\}$ is an event and $P[X \in A \cap L]=1$. Fix $v\in A\cap L$ and define a new random vector 
$$\tilde{X}= X 1_{\{X \in A \cap L\}} + v1_{\{X \notin A \cap L\}}$$
where $1_F$ is an indicator function that is $1$ if event $F$ is true, and $0$ else. 
Then $\tilde{X} \in A \cap L$ surely and $P[X=\tilde{X}]=1$, so $\expect{\tilde{X}}=\expect{X}=0$. 
Let $h:L\rightarrow\mathbb{R}^d$ be a linear bijection. In particular, $h(0)=0$. Since $A \cap L$ is a convex subset of $\mathbb{R}^n$, $h(A \cap L)$ is a convex subset of $\mathbb{R}^d$.  Since $h$ is a linear bijection and $0 \notin A$,  we know $0\notin h(A \cap L)$.  The hyperplane separation theorem ensures there is a separation between $0$ and the convex set $h(A\cap L)$, so there is a nonzero vector $a\in\mathbb{R}^d$ such that 
$$0\leq a^{\top}y \quad \forall y \in h(A \cap L)$$
Since $\tilde{X} \in A \cap L$ we know $h(\tilde{X}) \in h(A \cap L)$ and so 
$$ 0\leq a^{\top}h(\tilde{X})$$ 
On the other hand, linearity of $h$ implies 
$$\expect{a^{\top}h(\tilde{X})} = a^{\top}h\left(\expect{\tilde{X}}\right) = h(0)=0$$
Then $a^{\top}h(\tilde{X})$ is a nonnegative random variable with expectation zero, so 
\begin{equation} \label{eq:zero} 
P[a^{\top}h(\tilde{X})=0]=1
\end{equation} 
Recall $d\geq 1$. Define the $d-1$ dimensional linear subspace $C=\{y \in \mathbb{R}^d : a^{\top}y=0\}$.  
Then \eqref{eq:zero} means 
$$P[h(\tilde{X}) \in C] = 1$$
Therefore $P[\tilde{X} \in h^{-1}(C)]=1$. 
Since $X=\tilde{X}$ almost surely, we have 
$P[X \in h^{-1}(C)]=1$. Since $C$ is a $d-1$ dimensional subspace of $\mathbb{R}^d$ and $h:L\rightarrow\mathbb{R}^d$ is bijective, $h^{-1}(C)$ is a $d-1$ dimensional linear subspace of $\mathbb{R}^n$ such that $P[X \in h^{-1}(C)]=1$,
contradicting the definition of $d$ being the smallest integer for which this can hold. 
\end{proof} 

\

The assumption that $\expect{X}$ is finite is crucial. As a counter-example, let $X:\Omega\rightarrow\mathbb{R}$ be any random variable with $\expect{X}=\infty$. Then $X$ is in the convex set $\mathbb{R}$ surely, but $\expect{X} \notin \mathbb{R}$.

\subsection{Skewed distributions} 

This subsection presents a lemma that falls under the genre of  ``change of measure.'' The result is generally known; we provide a proof for completeness. Proofs of related statements often use the Radon-Nikodym derivative. Our proof uses more elementary concepts.

Let  $(\Omega, \script{F}, P)$ be a probability space. Fix positive integers $n,k$. 
Let $X:\Omega\rightarrow \mathbb{R}^n$ be a random vector. Let $T:\Omega\rightarrow[0,\infty)$ be a nonnegative random variable (possibly dependent on $X$).  Assume $0<\expect{T}<\infty$.
Let $Y$ be a random vector that takes values in $\mathbb{R}^n$ with distribution 
\begin{equation} \label{eq:skewer}
P[Y \in B] = \frac{\expect{T 1_{\{X \in B\}}}}{\expect{T}} \quad \forall B \in \script{B}(\mathbb{R}^n)
\end{equation}
It can be shown that this is a valid distribution.\footnote{Specifically, the function $h:\script{B}(\mathbb{R}^n)\rightarrow [0,1]$ defined by 
$h(B)=\expect{T1_{\{ X \in B\}}}/\expect{T}$ for $B \in \script{B}(\mathbb{R}^n)$
satisfies the three axioms for a probability measure for the measurable space $(\mathbb{R}^n, \script{B}(\mathbb{R}^n))$. Indeed, $h(B)\geq 0$ for all $B \in \script{B}(\mathbb{R}^n)$; $h(\mathbb{R}^n)=1$; $h(\cup_{i=1}^{\infty} B_i)=\sum_{i=1}^{\infty} h(B_i)$ for disjoint sets $B_i \in \script{B}(\mathbb{R}^n)$. The identity random vector $W:\mathbb{R}^n\rightarrow\mathbb{R}^n$ defined by $W(\omega)=\omega$ for all $\omega \in \mathbb{R}^n$ has this distribution.} 

\

\begin{lem} \label{lem:skew} For any Borel measurable function $g:\mathbb{R}^n\rightarrow\mathbb{R}^k$ such that $\expect{Tg(X)}$ is finite, we have 
$$ \expect{g(Y)} = \frac{\expect{Tg(X)}}{\expect{T}}$$
\end{lem} 

\

\begin{proof} 
First assume $g:\mathbb{R}^n\rightarrow[0, \infty)$ is a nonnegative function. 
Then $g(Y)$ is a nonnegative random variable and
\begin{align*}
\expect{g(Y)} &=\int_0^{\infty} P[g(Y)>t]dt\\
&=\int_0^{\infty} P[Y\in g^{-1}((t,\infty))]\\
&\overset{(a)}{=}\int_0^{\infty}\frac{\expect{T1_{\{X \in g^{-1}((t,\infty))\}}}}{\expect{T}}dt\\
&=\frac{1}{\expect{T}}\int_0^{\infty} \expect{T1_{\{g(X)>t\}}}dt\\
&\overset{(b)}{=}\frac{1}{\expect{T}}\expect{T \int_0^{\infty} 1_{\{g(X)>t\}}dt} \\
&\overset{(c)}{=}\frac{1}{\expect{T}}\expect{Tg(X)}
\end{align*}
where (a) holds by \eqref{eq:skewer};  (b) holds by the Fubini-Tonelli theorem; (c) uses the fact that $g(X)\geq 0$ so
$$ \int_0^{\infty} 1_{\{g(X)>t\}}dt = g(X)$$

This extends to real-valued functions $g:\mathbb{R}^n\rightarrow\mathbb{R}$ by defining $g_1, g_2$ as the positive and negative parts, so $g=g_1-g_2$. If  $\expect{g_1(Y)}$ and $\expect{g_2(Y)}$ are both finite then $\expect{g(Y)}=\expect{g_1(Y)}-\expect{g_2(Y)}$.  Finally, functions of the form $g=(g_1, \ldots, g_k)$ yield $\expect{g(Y)} = (\expect{g_1(Y)}, \ldots, \expect{g_k(Y)})$. 
\end{proof}

\bibliographystyle{unsrt}
\bibliography{../../../latex-mit/bibliography/refs}
\end{document}